\documentclass[10pt]{article}
\usepackage{amsmath,amssymb,amsfonts,amsthm}
\usepackage{mathrsfs}
\usepackage{indentfirst}
\usepackage[pdftex]{color,graphicx}
\usepackage[pdftex,bookmarks,unicode,colorlinks]{hyperref}
\usepackage{enumerate} 
\usepackage[numbers]{natbib} 
\usepackage{yfonts} 
\usepackage[title]{appendix} 

\makeatletter
\let\@fnsymbol\@arabic 
\makeatother

\theoremstyle{plain}
\newtheorem{theorem}{Theorem}[section]
\newtheorem{proposition}[theorem]{Proposition}
\newtheorem{corollary}[theorem]{Corollary}
\newtheorem{lemma}[theorem]{Lemma}

\theoremstyle{remark}
\newtheorem{remark}[theorem]{\underline{Remark}}
\newtheorem{example}[theorem]{\underline{Example}}
\newtheorem{condition}{Cond.}

\theoremstyle{definition}
\newtheorem{definition}[theorem]{Definition}

\newcommand{\E}{\mathbf E}
\renewcommand{\P}{\mathbf P}

\newcommand{\R}{\mathbb R}
\newcommand{\N}{\mathbb N}

\newcommand{\Ro}[1]{\mathbb R^{#1}\setminus \{0\}}
\newcommand{\C}{\mathcal C}
\newcommand{\D}{\mathcal D}

\newcommand{\e}{\epsilon}
\newcommand{\F}{\mathcal F}
\newcommand{\B}{\mathcal B}
\newcommand{\X}{\mathcal X}
\newcommand{\Pred}{\mathcal P}
\newcommand{\ind}{\mathbf 1}

\numberwithin{equation}{section} 

\begin{document}

\title{\bf\Large Large Deviations for Stochastic Differential Equations \\ Driven by Semimartingales}
\author{\bf\normalsize{
Qiao Huang\footnote{Group of Mathematical Physics (GFMUL), Department of Mathematics, Faculty of Sciences, University of Lisbon, Campo Grande, Edif\'{\i}cio C6, PT-1749-016 Lisboa, Portugal. Email: \texttt{qhuang@fc.ul.pt}} $^,$\footnote{Present Address: Division of Mathematical Sciences, School of Physical and Mathematical Sciences, Nanyang Technological University, 21 Nanyang Link, Singapore 637371. Email: \texttt{qiao.huang@ntu.edu.sg}},
Wei Wei\footnote{Institute of Natural Sciences, Shanghai Jiao Tong University, 800 Dongchuan RD. Minhang District, Shanghai 200240, P.R.China Email: \texttt{weiw\_sjtu@sjtu.edu.cn}},
Jinqiao Duan\footnote{Department of Mathematics and Department of Physics, Great Bay University, Dongguan, Guangdong 523000 P.R.China; Email: \texttt{duan@gbu.edu.cn}}
}}

\date{}
\maketitle
\vspace{-0.3in}

\begin{abstract}
  We prove a large deviation principle for stochastic differential equations driven by semimartingales, with additive controls. Conditions are given in terms of characteristics of driven semimartingales, so that if the noise-control pairs satisfy a large deviation principle with some good rate function, so do the solution processes. There is no joint exponential tightness assumption for noise-control-solution triplets and no uniform exponential tightness assumption for noise.
\end{abstract}

\section{Introduction}

The theory of large deviations is concerned with asymptotic computation of probabilities of rare events on an exponential scale. In its basic form, the theory considers the limit of normalizations of logarithmic probability for a sequence of events with asymptotically vanishing probability. It has proved to be an crucial tool to tackle many problems in statistics, engineering, statistical mechanics, and applied probability. The sample path large deviations focus mainly on the Markovian solutions of small noise stochastic differential equations (SDEs) driven by Brownian motions \cite{BDM08, DZ98, FK06} and Poisson jump measures \cite{BDM11, BCD13}.

Semimartingales may fail to have the Markov property. In general, the large deviation principles for semimartingales are formulated on the \emph{Skorokhod space} $\D_d:=\D(\R_+;\R^d)$, which is the space of all $\R^d$-valued c\`adl\`ag functions on $\R_+$, equipped with the Skorokhod topology. In \cite{Puh94}, A.~Puhalskii established conditions for the large deviation principle in the Skorohod topology to hold for a sequence of semimartingales in terms of the convergence of their predictable characteristics. In that paper and in his follow-up work \cite{Puh94b}, Puhalskii derived many parallels between exponential convergence in the form of large deviations and weak convergence, which are formulated more systematically in his monograph \cite{Puh01} via the theory of idempotent measures. We will draw on this idea in this current paper.

In this paper, we are concerned with the large deviation principle (LDP) for the family of solution processes of semimartingale-driven stochastic differential equations. For each $\e>0$, suppose we have a filtered probability space $(\Omega^\e,\F^\e,\{\F^\e_t\}_{t\ge0},\P^\e)$ with right-continuous filtration, endowed with a $d$-dimensional c\`adl\`ag semimartingale $X^\e$ and an $n$-dimensional c\`adl\`ag adapted process $U^\e$. We also have a global Lipschitz function $F: \R^n \to \R^{n\times d}$, so that each stochastic differential equation
\begin{equation}\label{sde}
  Y^\e_t = U^\e_t + \int_0^t F(Y^\e_{s-}) d X^\e_s,
\end{equation}
has a unique global solution $Y^\e$ which is an $n$-dimensional $\{\F^\e_t\}$-adapted process with ($\P^\e$-a.s.) c\`adl\`ag paths (see, e.g., \cite{Pro05}). 
The symbol $F_-$ is used to denote the process $F_{t-}:= \lim_{s\uparrow t} F_s$.

Logically, system \eqref{sde} has ``inputs'' the \emph{control} $U^\e$ and the \emph{noise} $X^\e$, and ``output'' the \emph{solution} $Y^\e$. Follow the principle of causality for stochastic dynamical systems (cf. \cite[Section 5.2]{KS91}), a natural \emph{question} is: if the family of ``inputs'' $\{(X^\e,U^\e)\}_{\e>0}$ satisfies a large deviation principle, is it true that the family of ``outputs'' $\{Y^\e\}_{\e>0}$ also satisfies a large deviation principle?

A similar question has been investigated in \cite{Gan18,Gar08}. The former considered the case that $U^\e \equiv 0$ (which is not essential), and the latter concerned itself with the infinite-dimensional case. In both papers, a \emph{uniform exponential tightness} (UET) condition on the family of noise $\{X^\e\}_{\e>0}$ was proposed to prove the LDP for the solution family $\{Y^\e\}_{\e>0}$, provided that an LDP holds for $\{X^\e\}_{\e>0}$ and that $\{(X^\e,U^\e,Y^\e)\}_{\e>0}$ is \emph{exponentially tight}.
In the presence of exponential tightness of the family $\{(X^\e,U^\e,Y^\e)\}_{\e>0}$, the classical theory of large deviations tells that this family has a subsequence that satisfies an LDP with some good rate function (see, for instance, \cite[Lemma 4.1.23]{DZ98}). Therefore, to prove the LDP for this family, one only needs to assume it holds and then to identify the rate function, and finally to show that this rate function does not depend on the choice of subsequences.

The results in those two papers partially answered preceding question, as the assumption of joint exponential tightness of $\{(X^\e,U^\e,Y^\e)\}_{\e>0}$ involves additional condition on the ``outputs'' $Y^\e$'s. Although this joint exponential tightness may not be hard to verify from the exponential tightness of $\{(X^\e,U^\e)\}_{\e>0}$ and $\{Y^\e\}_{\e>0}$ in the case that characteristics of $X^\e$ are deterministic (cf. \cite[Theorem 4.1]{FK06}), we still think that it has limitations and does not quite fit the principle of causality. For these reasons, we will discard the exponential tightness assumption for $\{(X^\e,U^\e,Y^\e)\}_{\e>0}$ in this paper. The trick is to separate the joint exponential tightness of ``inputs'' $\{(X^\e,U^\e)\}_{\e>0}$ and ``outputs'' $\{Y^\e\}_{\e>0}$ via a S{\l}omi\'nski-type criterion (Proposition \ref{Slominski}) which has its own practicability. See also Step 3 of the proof of Theorem \ref{LDP-special} below. The S{\l}omi\'nski criterion was firstly proposed in \cite[Proposition 2]{Slo89} to separate the joint tightness as an intermediate step to prove the stability of weak convergence of SDEs. What we do in Proposition \ref{Slominski} is to generalize this kind of criterion to the context of exponential tightness and large deviation principles. Meanwhile, the UET condition proposed in \cite{Gan18,Gar08} is aiming to extend weak convergence results for stochastic integrals to the large deviation setting.
It is an exponential version of uniform tightness (UT) condition in the study of weak convergence (cf. \cite[Section VI.6]{JS13}). We will also get rid of the UET condition, and give various sufficient conditions for the UET, in terms of characteristics of semimartingales.

The main result of this paper is to answer the preceding question in a complete fashion, that is, we are only 
provided the exponential tightness of noise-control pairs $(X^\e,U^\e)$'s. All assumptions are made only in terms of the characteristics of driving noise $X^\e$, which is easy to check in practice. For each $f\in\D_{n\times d}$ and $x\in \D_d$ with $x$ of finite variation locally, we denote $(f\cdot x)(t):=\lim_{\|\Delta\|\to 0} \sum_{i} f(t_i) (x(t_{i+1}) - x(t_i))$, where $\Delta = \{t_i\}_{i=0}^k$ is a partition of the interval $[0,t]$ with $0=t_0\le t_1\le \cdots \le t_k=t$, and $\|\Delta\| := \max_{1\le i\le k}|t_{i+1}-t_i|$ denotes the mesh of $\Delta$. For an adapted process $B$ with locally finite variation, we denote by $V(B)$ its \emph{variation process}, namely, $V(B)_t(\omega)$ is the total variation of the function $s\mapsto B_s(\omega)$ on the interval $[0,t]$. We call $h:\R^d \to \R^d$ a \emph{truncation function} if it is bounded and satisfies $h(x) = x$ in a neighborhood of 0. For $b>0$, we define a truncation function $h_b(x):= x\ind_{\{|x|\le b\}}$, which is commonly used. With respect to a given truncation function $h$, we can associate each $X^\e$ a triplet $(B^\e(h),C^\e,\nu^\e)$ which is called \emph{characteristics}. That is, $B^\e(h)$ is the process with locally finite variation in the special semimartingale part of $X^\e$ that has jumps dominated by $h$, $C^{\e} $ is the quadratic variation of the continuous martingale part and $\nu^\e$ the L\'evy system of $X^{\e}$. See Section \ref{sec-3} for more details on these notions and notations. Then we will show that

\begin{theorem}\label{LDP-special}
  Let $F$ be a bounded Lipschitz function with both Lipschitz constant and itself bounded by constant $c>0$.
  For each $\e>0$, let $X^\e$ be a quasi-left-continuous semimartingales with characteristics $(B^\e(h),C^\e,\nu^\e)$ associated to the truncation function $h$, and let $Y^\e$ be the solution of \eqref{sde}. 
  Assume there exist $b>0$ and $0<r<\frac{1}{4c}$ such that for each $t>0$, the following three real-valued families are all exponentially tight,
  \begin{equation}\label{three-families}
    \{V(B^\e(h_b))_t\}_{\e>0}, \quad \left\{\frac{1}{\e} C^\e_t\right\}_{\e>0}, \quad \left\{ \e \int_{\Ro d} \exp\left(\frac{|x|}{\e r}\vee 1 \right) \nu^\e(dx,[0,t]) \right\}_{\e>0}, 
  \end{equation}
  If the family $\{(X^\e,U^\e)\}_{\e>0}$ satisfies the LDP with good rate function $I'$ on $\D_{d+n}$, then the family $\{(X^\e,U^\e,Y^\e)\}_{\e>0}$ also satisfies the LDP with the following good rate function on $\D_{d+2n}$,
  \begin{equation}\label{rate-func}
    I(x,u,y) =
    \begin{cases}
      I'(x,u), & y = u+F(y)\cdot x, \ x \text{ is locally of finite variation}, \\
      \infty, & \text{otherwise}.
    \end{cases}
  \end{equation}
  In particular, the family $\{Y^\e\}_{\e>0}$ satisfies the LDP with the following good rate function
  \begin{equation}\label{rate-func-2}
    I^\flat(y) =\inf\{ I'(x,u): y = u+F(y)\cdot x, \ x \text{ is locally of finite variation}\}.
  \end{equation}
\end{theorem}

A direct application of our theorem is to investigate the large deviations of SDEs driven by \emph{exponentially integrable} additive process, as illustrated in Example \ref{examp}. The solutions to SDE driven by L\'evy noise in both finite-dimensional and infinite-dimensional settings, have been showed to obey an LDP in \cite{BDM11, BCD13}, via the weak convergence approach. The exponential tightness condition for the third family in \eqref{three-families} is a generalization of the exponential integrability condition in the context of L\'evy noise. But to the best of our knowledge, there is no literature applying the weak convergence approach to SDEs driven by general semimartingales and seeking conditions for characteristics.

The sequel of this paper is organized as follows. In the next section, we will recall definitions of large deviation principles and exponential tightness. Then we establish a S{\l}omi\'nski-type criterion for the exponential tightness in Skorokhod space. In Section \ref{sec-3}, we will recall some basic notions in the theory of semimartingales and give various characterizations for the UET property of semimartingales. The relations between UET property and exponential tightness will be discussed as well. Section \ref{sec-4} will be devoted to the proofs of our main result Theorem \ref{LDP-special}. Finally, Section \ref{sec-5} is reserved an application to SDE driven by additive noise. More characterizations for the UET property and a few comments on $\C$-exponential tightness are left into Appendix \ref{app} and \ref{app-2}.

\section{Exponential tightness}\label{sec-2}

Let $\X$ be a topological space with countable base, endowed with Borel $\sigma$-algebra $\B(\X)$. A \emph{good rate function} $I$ is a lower semicontinuous mapping $I:\X\to [0,\infty]$ such that for all $\alpha\in [0,\infty)$, the level set $\Phi_I(\alpha):=\{x\in\X: I(x)\le\alpha\}$ is a closed, \emph{compact} subset of $\X$. A family of probability measures $\{\mu_\e\}_{\e>0}$ on $(\X,\B(\X))$ is said to satisfy the \emph{large deviation principle} (LDP) with a good rate function $I$ if, for all $\Gamma\in\B(\X)$, one have
\begin{equation*}
  -\inf_{x\in\Gamma^\circ} I(x) \le \liminf_{\e\to0} \e\log \mu_\e(\Gamma) \le \limsup_{\e\to0} \e\log \mu_\e(\Gamma) \le -\inf_{x\in\overline\Gamma} I(x),
\end{equation*}
where $\Gamma^\circ$ and $\overline\Gamma$ denote the topological interior and closure of $\Gamma$. A family of $\X$-valued random elements $\{X^\e\}_{\e>0}$ is said to satisfies a large deviation principle if the family of probability measures induced by $X^\e$ on $\X$ satisfies a large deviation principle. We refer to \cite{DZ98,DE11} for more details of the large deviation theory.

A family of probability measure $\{\mu_\e\}_{\e>0}$ on $(\X,\B(\X))$ is said to be \emph{exponentially tight} if for every $\alpha<\infty$, there exists a compact set $K \subset \X$ such that
\begin{equation*}
  \limsup_{\e\to0} \e \log \mu_\e(K^c) < -\alpha.
\end{equation*}
Or equivalently, for every $0<\delta<1$, there exists a compact set $K\subset \X$ and $\e_0>0$, such that for all $0<\e<\e_0$,
\begin{equation*}
  [\mu_\e(K^c)]^\e < \delta.
\end{equation*}
A family of $\X$-valued random elements $\{X^\e\}_{\e>0}$ is said to exponentially tight if the family of induced probability measures on $\X$ is exponentially tight. In the case of $\X=\R^d$, the exponential tightness is equivalent to the exponential stochastic boundedness found in \cite[Definition 3.2]{Gar08}.

It is well known that the Skorokhod space $\D_d$ is a Polish space. Hence, the exponential tightness of a family of probability measures on $(\D_d,\B(\D_d))$ is implied by the LDP with good rate function (see, e.g., \cite[Section 1.2]{DZ98}). For $\rho>0$ and $T>0$, denote by $\Delta_\rho[0,T]$ the set of all partitions $\{t_i\}_{i=0}^k$, $k\in\N_+$, such that $0\le t_0\le t_1\le \cdots \le t_k = T$ and $t_j-t_{j-1}>\rho$ for all $j=1,2,\cdots,k-1$. For $x = \{x(t)\}_{t\ge0}\in \D_d$, define
\begin{align*}
  w(x,I) &= \sup_{s,t\in I} |x(s)-x(t)|, \quad \text{for } I \text{ an interval of } \R_+, \\
  w_T(x,\rho) &= \inf_{\{t_i\}\in\Delta_\rho[0,T]} \max_{1\le i\le k} w(x,[t_{i-1},t_i)).
\end{align*}

The following proposition is a criterion for the exponential tightness of probability measures in $\D_d$, referring to \cite[Theorem 4.2]{Puh91}.
\begin{lemma}\label{exp-tight}
  A family of probability measure $\{\mu_\e\}_{\e>0}$ on $\D_d$ is exponentially tight if and only if

  (i). for all $T>0$,
  \begin{equation*}
    \lim_{a\to\infty} \limsup_{\e\to0} \e\log \mu_\e\left( x\in \D_d: \sup_{0\le t\le T}|x(t)|\ge a \right) = -\infty,
  \end{equation*}

  (ii). for all $T>0$ and $\eta>0$,
  \begin{equation*}
    \lim_{\rho\to0} \limsup_{\e\to0} \e \log \mu_\e\left( x\in \D_d: w_T(x,\rho) \ge\eta \right) = -\infty.
  \end{equation*}
\end{lemma}
For each $\e>0$ we have a $d$-dimensional c\`adl\`ag process $X^\e$ on a filtered probability space $(\Omega^\e,\F^\e,\{\F^\e_t\}_{t\ge0},\P^\e)$. Let $\mu_\e = \P^\e \circ (X^\e)^{-1}$ be the probability measure on $\D_d$ induced by $X^\e$. In this case, the condition (i) in Lemma \ref{exp-tight} is equivalent to say that for each $T>0$, the family of random variables $\{\sup_{0\le t\le T}|X^\e_t|\}_{\e>0}$ is exponential tight, which is referred to as the exponential compact containment condition in \cite[Remark 4.5]{FK06}.

The following proposition is an analog of the S{\l}omi\'nski's criterion for tightness (\cite[Proposition 2]{Slo89} or \cite[Section VI.7]{JS13}). It will be used to separate the joint exponential tightness, and play a role in proving the main result.

\begin{proposition}[S{\l}omi\'nski-type criterion for exponential tightness]\label{Slominski}
  (i). Assume the family of probability measure $\{\mu_\e = P^\e \circ (X^\e)^{-1}\}_{\e>0}$ satisfies Lemma \ref{exp-tight}.(i). If there exists a family of strictly increasing sequences of random times $\{T^{\e,p}_i\}_{i=0}^\infty$ for each $\e>0, p\in\N_+$ with $T^{\e,p}_0=0$ and $\lim_{i\to\infty}T^{\e,p}_i = \infty$, and strictly positive constants $\rho_p^N$ for each $p,N\in\N_+$, such that
  \begin{equation}\label{exp-tight-1}
    \lim_{p\to\infty} \limsup_{\e\to0} \e\log \P^\e \left( \inf_{i:T^{\e,p}_{i+1}\le N} \left(T^{\e,p}_{i+1}-T^{\e,p}_i \right) \le \rho_p^N \right) = -\infty, \text{ for all } N>0,
  \end{equation}
  and
  \begin{equation}\label{exp-tight-2}
    \lim_{p\to\infty} \limsup_{\e\to0} \e\log \P^\e \left( \sup_{i} w\left(X^\e,[T^{\e,p}_i,T^{\e,p}_{i+1}) \cap [0,N]\right) \ge \eta \right) = -\infty,  \text{ for all } N>0, \eta>0,
  \end{equation}
  then $\{X^\e\}_{\e>0}$ is exponentially tight.

  (ii). If the family $\{X^\e\}_{\e>0}$ satisfies an LDP with a good rate function, then Lemma \ref{exp-tight}.(i) holds. Moreover, there are positive constants $\rho_p^N$ and $a(\e,p,i)\in(\frac{1}{2p},\frac{1}{p}]$, such that \eqref{exp-tight-1} and \eqref{exp-tight-2} are satisfied by the following stopping times, defined recursively on $i$ by $T^{\e,p}_0 = 0$ and
  \begin{equation}\label{random-times}
    T^{\e,p}_{i+1} = \inf\left\{t>T^{\e,p}_i: |X^\e_t - X^\e_{T^{\e,p}_i}| \ge a(\e,p,i) \text{ or } |X^\e_{t-} - X^\e_{T^{\e,p}_i}| \ge a(\e,p,i) \right\}.
  \end{equation}
\end{proposition}
\begin{proof}

  (i). Denote $T(\e,p,N):=\inf_{i:T^{\e,p}_{i+1}\le N} \left(T^{\e,p}_{i+1}-T^{\e,p}_i \right)$ and
  $$A(\e,p,N,\eta):=\left\{ \sup_{i} w\left(X^\e,[T^{\e,p}_i,T^{\e,p}_{i+1}) \cap [0,N]\right) \ge \eta \right\}.$$
  Fix $N>0$, $\eta>0$ and $\delta>0$. By \eqref{exp-tight-1} and \eqref{exp-tight-2}, there exist $p_0>0$ and $\e_0>0$ such that for all $p\ge p_0$ and $0<\e\le\e_0$,
  \begin{equation*}
    \left[ \P^\e\left(T(\e,p,N)\le \rho_p^N\right)\right]^\e < \delta/2, \quad \left[ \P^\e(A(\e,p,N,\eta))\right]^\e < \delta/2.
  \end{equation*}
  Let $\e_1 = \e_0 \wedge 1$. Then for all $0<\e\le\e_1$,
  \begin{equation}\label{est-1}
    \left[ \P^\e\left( \left\{T(\e,p,N)\le \rho_p^N \right\} \cup A(\e,p,N,\eta)\right)\right]^\e < \delta.
  \end{equation}
  For each $\omega\in \left\{T(\e,p,N)> \rho_p^N \right\} \cap A(\e,p,N,\eta)^c$, denoted $k(\omega):=\max\{i: T^{\e,p}_{i+1} \le N\}$. Then for all $0\le i\le k$, $T^{\e,p}_{i+1}-T^{\e,p}_i > \rho_p^N$ and $w\left(X^\e,[T^{\e,p}_i,T^{\e,p}_{i+1})\right) < \eta$. Hence, $\omega\in \{w_N(X^\e,\rho_p^N) < \eta\}$. Due to \eqref{est-1}, we have
  \begin{equation*}
    \left[ \P^\e\left( w_N(X^\e,\rho_p^N) \ge \eta \right)\right]^\e < \delta.
  \end{equation*}
  The exponential tightness follows from Lemma \ref{exp-tight}.

  (ii). Assume the family $\{X^\e\}$ to satisfy an LDP with a good rate function. Define for $a>0$ and $x\in\D_d$,
  \begin{align*}
    S_a(x) &:= \inf\{t\ge0: |x(t)|\ge a \text{ or } |x(t-)|\ge a\}, \\
    S_{a+}(x) &:= \inf\{t\ge0: |x(t)|> a \text{ or } |x(t-)|> a\}, \\
    V(x) &:= \{a>0: S_a(x)<S_{a+}(x)\}, \\
    V'(x) &:= \{a>0: \Delta x(S_a(x))\ne0, |x(S_a(x)-)| = a \}.
  \end{align*}
  We will construct $a(\e,p,i)$ and $\rho_p^N$ using the random times $T^{\e,p}_{i}$ defined in \eqref{random-times}. Set
  \begin{equation*}
    U(\e,p,i):= V(X^\e-X^\e_{\cdot \wedge T^{\e,p}_i}) \cup V'(X^\e-X^\e_{\cdot \wedge T^{\e,p}_i}),
  \end{equation*}
  Then using \cite[Lemma VI.2.10]{JS13}, we know that for each $\omega\in\Omega^\e$, $U(\e,p,i)(\omega)$ is an at most countable subset of $\R_+$. It follows that the set
  \begin{equation*}
    V(\e,p,i) := \left\{ a>0: \P^\e\left( a\in U(\e,p,i) \right) =0 \right\}
  \end{equation*}
  has full measure in $\R_+$. So we can choose $a(\e,p,i)\in V(\e,p,i) \cap (\frac{1}{2p},\frac{1}{p}]$. By the definition \eqref{random-times} of $T^{\e,p}_i$, we know that $w\left(X^\e,[T^{\e,p}_i,T^{\e,p}_{i+1})\right) \le 2a(\e,p,i) \le 2/p$. Then for each $N$ and $\eta$,
  \begin{equation*}
    \lim_{p\to\infty} \limsup_{\e\to0} \e\log \P^\e( A(\e,p,N,\eta) ) \le \lim_{p\to\infty} \limsup_{\e\to0} \e\log \P^\e( 2/p \ge \eta ) = -\infty.
  \end{equation*}
  And \eqref{exp-tight-2} follows.

  Using the above notations, we have $T^{\e,p}_{i+1} = S_{a(\e,p,i)}(X^\e-X^\e_{\cdot \wedge T^{\e,p}_i})$. Applying induction in $i$ and \cite[Proposition VI.2.11, VI.2.12]{JS13}, it is easy to see that for each $p$, the mapping
  \begin{equation*}
    X^\e \mapsto (T^{\e,p}_0,T^{\e,p}_1,\cdots,T^{\e,p}_i,\cdots)
  \end{equation*}
  is continuous. And hence, for each $p$ and $N$, the mapping
  \begin{equation*}
    X^\e \mapsto T(\e,p,N):=\inf_{i:T^{\e,p}_{i+1}\le N} \left(T^{\e,p}_{i+1}-T^{\e,p}_i \right)
  \end{equation*}
  is also continuous. Now the contraction principle in the theory of LDP (see, e.g., \cite[Theorem 4.2.1]{DZ98}) tell us that the family of random times $\{T(\e,p,N)\}_{\e>0}$ satisfies an LDP with some good rate function, say $I_{p,N}:(0,N]\to[0,\infty]$ for each $p$ and $N$.

  We now show that $\lim_{t\to0^+} I_{p,N}(t) = \infty$ by contradiction, using the goodness of $I_{p,N}$. Suppose conversely that $\liminf_{t\to0^+} I_{p,N}(t) = L$ for some $L<\infty$. Then there is a subsequence $\{t_i\}_{i=1}^\infty$ such that $t_i \downarrow 0$ and $I_{p,N}(t_i)\to L$ as $i\to\infty$. It follows that for each $\delta>0$, $\{t_i\}_{i=M}^\infty \subset \Phi_{I_{p,N}}(L+\delta)$ for some $M\in\N_+$, where we use $\Phi$ as before to denote the level set. This produce a contradiction, since each level set of $I_{p,N}$ is compact due to the goodness.

  Hence, for each $N$, we can select a sequence of positive reals $\{\rho_p^N\}_{p=1}^\infty$ which goes to 0 as $p\to\infty$, such that
  $$\inf_{t\le \rho_p^N} I_{p,N}(t)\ge p.$$
  Then
  \begin{equation*}
    \lim_{p\to\infty} \limsup_{\e\to0} \e\log \P^\e \left( T(\e,p,N) \le \rho_p^N \right) \le - \lim_{p\to\infty} \inf_{t\le \rho_p^N} I_{p,N}(t) = -\infty,
  \end{equation*}
  which yields \eqref{exp-tight-1}.
\end{proof}

\section{Uniform exponential tightness}\label{sec-3}

In this section, we will revisit the notion of uniform exponential tightness proposed first in \cite{Gar08}. It is an analogy of predictable uniform tightness in the context of weak convergence (see, e.g., \cite{JS13,Slo89}). We will seek various sufficient conditions for the uniform exponential tightness for local martingales and processes with locally finite variation. Combining these conditions, we will obtain a useful criterion (Proposition \ref{exp-tight-general}) for the uniform exponential tightness of general semimartingales, in terms of their characteristics. As by-products, the relations between the uniform exponential tightness of processes and the exponential tightness of their functionals are developed as much as we can, which may play a role in future works.


Firstly, let us recall some basic notions in the theory of semimartingales, referring to \cite[Chapter II]{JS13} and \cite[Chapter IX]{HWY92}. A \emph{semimartingale} $X$ is a process of the form $X = B + M$, where $B$ is a c\`adl\`ag adapted process with locally finite variation and $M$ is a local martingale with $M_0=0$. This decomposition is of course not unique. However, it is unique (in the indistinguishable sense) when $B$ is predictable. In this case, the semimartingale $X$ is called \emph{special} and the unique decomposition is called the \emph{canonical decomposition}. The \emph{quadratic variation} (sharp bracket) of semimartingale $X$ is denoted by $[X,X]$. If $M$ is a locally square-integral martingale, we denote by $\langle M,M\rangle$ its \emph{predictable quadratic variation} (angle bracket). We use the symbol $\Delta X$ to denote the process $\Delta X_t := X_t - X_{t-}$.

For a fixed truncation function $h$ and a semimartingale $X$ (which is not necessarily special), define $\check X(h): = \sum_{s\le\cdot}(\Delta X_s - h(\Delta X_s))$ and $X(h): = X - \check X(h)$. Then $X(h)$ is obviously a special semimartingale, which has a unique decomposition, say $X(h) = B(h) + M(h)$. We call the decomposition $X = \check X(h) + B(h) + M(h)$ the \emph{canonical decomposition associated to} $h$ for the semimartingale $X$. The local martingale $M(h)$ admits a unique (up to indistinguishability) decomposition $M(h) = M(h)^c + M(h)^d$, where $M(h)^c$ is a continuous local martingale and $M(h)^d$ is a purely discontinuous local martingale. There is a unique (up to indistinguishability) continuous local martingale $X^c$ such that any canonical decomposition associated to $h$ for $X$ meets $X^c = M(h)^c$. The process $X^c$, which is independent of the choice of $h$, is called \emph{continuous martingale part} of $X$. We can associated to a c\`adl\`ag adapted process $X$ a integer-valued random measure $\mu^X(dx,dt) := \sum_s\ind_{\{\Delta X_s\ne 0\}} \delta_{(\Delta X_s,s)}(dx,dt)$, which we call the \emph{jump measure} of $X$. The dual predictable projection (compensator) of $\mu^X$ is called the \emph{L\'evy system} of $X$. A semimartingale $X$ is said to be \emph{quasi-left-continuous} if there exists a version of its L\'evy system $\nu^X$ that satisfies identically $\nu^X(\R^d\times\{t\}) = 0$. We call \emph{characteristics} associated with $h$ of $X$ the triplet $(B(h),C,\nu^X)$, with $C = \langle X^c,X^c\rangle$ and $\nu^X$ the L\'evy system of $X$.

For a semimartingale $X$ and an adapted process $F$, we denote by $F_- \cdot X$ the stochastic integral $\int_0^\cdot F_{s-} dX_s$. For a random field $G$ on $\R_+ \times \R^d$ and a random measure $\mu$, we denote by $G*\mu$ the integral process $G*\mu_t:=\int_0^t\int_{\R^d} G(s,x)\mu(dx,ds)$, if the integral is well-defined.

The following exponential estimate for purely discontinuous local martingales will be useful.
\begin{lemma}\label{pure-disc}
  Let $M$ be an one-dimensional purely discontinuous local martingale starting at $0$, satisfying $|\Delta M| \le A$ with some constant $A>0$. Let $\mu^M$ be the jump measure of $M$. Then there exists a constant $c>0$, such that for all $t>0$, and every $a,b>0$,
  \begin{equation*}
    \P\left( \sup_{0\le s\le t}|M_s|\ge a, |x|^2 * \mu^M_t <b \right) \le 2\exp\left(-\frac{a^2}{4cb}\right).
  \end{equation*}
\end{lemma}
\begin{proof}
  Let $\theta$ be a constant less than $\frac{1}{2A}$. Then $|\Delta(\theta M)| \le \frac{1}{2}$ and the random measure $\mu^{\theta M}$ associated to $\theta M$ is supported in the closed ball $\bar B(0,\frac{1}{2})$. Since $\theta M$ is a purely discontinuous local martingale, its stochastic exponential
  \begin{equation*}
    \mathcal E(\theta M)_t := \exp\left( \theta M_t- (x-\log(1+x)) * \mu^{\theta M}_t \right)
  \end{equation*}
  is a local martingale (see \cite[Theorem I.4.61]{JS13}). Define a stopping time $T = \inf\{t\ge0: |x|^2 * \mu^M_t \ge b\}$. Then the stopped processes $\mathcal E(\theta M)_{\cdot\wedge T} = \mathcal E(\theta M_{\cdot\wedge T})$ is a martingale. It is an easy fact that there exists a constant $c>0$, such that
  \begin{equation*}
    |x-\log(1+x)| \le c|x|^2, \quad\text{for all } |x|\le 1/2.
  \end{equation*}
  Using Doob's martingale inequality, we have
  \begin{equation*}
    \begin{split}
      \P\left( \sup_{0\le s\le t}|M_s|\ge a, |x|^2 * \mu^M_t <b \right) =&\ \P\left( \sup_{0\le s\le t}|\theta M_s|\ge \theta a, c|x|^2 * \mu^{\theta M}_t <c\theta^2b \right) \\
      \le&\ \P\left( \sup_{0\le s\le t}|\theta M_s|\ge \theta a, |x-\log(1+x)| * \mu^{\theta M}_t <c\theta^2b \right) \\
      \le&\ \P\left( \sup_{0\le s\le t} \mathcal E(\theta M)_{s\wedge T} \ge e^{\theta a - c\theta^2 b}\right) + \P\left( \sup_{0\le s\le t} \mathcal E(-\theta M)_{s\wedge T} \ge e^{\theta a - c\theta^2 b}\right) \\
      \le&\ 2 e^{-\theta a + c\theta^2 b}.
    \end{split}
  \end{equation*}
  By choosing $\theta \le \frac{1}{2A} \wedge \frac{a}{2cb}$, the result follows.
\end{proof}

The following lemma is a small adaption of \cite[Lemma 5.2]{Puh94}. We will omit the proof.
\begin{lemma}\label{estimates}
  Let $\mu$ be an integer-valued random measure and $\nu$ be its compensator. Let $G$ be a predictable random field on $\R_+ \times \R^d$ such that $G*\mu$ is a locally integrable increasing process. Then for every stopping time $T$ and all $a>0$ and $b>0$,
  \begin{align*}
    \P\left( \sup_{0\le s\le T}G*\mu_s >a \right) &\le e^{b-a} + \P\left( \sup_{0\le s\le T} \left(e^G-1\right)*\nu_s >b \right).
  \end{align*}
\end{lemma}

Now we investigate the uniform exponential tightness for semimartingales. Recall that for each $\e>0$, let $(\Omega^\e,\F^\e,\{\F^\e_t\}_{t\ge0},\P^\e)$ be a filtered probability space with right-continuous filtration. Let $\Pred^\e$ be the collection of simple $d$-dimensional $\{\F^\e_t\}$-adapted \emph{predictable} processes and
\begin{equation*}
  \Pred^\e_1 := \left\{ H\in\Pred^\e: \sup_{t\ge0}|H_t| \le 1 \right\}.
\end{equation*}
\begin{definition}\label{UET-def}
  A family $\{X^\e\}_{\e>0}$ of adapted c\`adl\`ag $d$-dimensional processes is said to be \emph{uniformly exponentially tight} (UET) if for every $t>0$,
  \begin{equation*}
    \lim_{a\to\infty} \limsup_{\e\to0} \sup_{H\in \Pred^\e_1} \e\log\P^\e \left( \sup_{0\le s\le t}\left|\sum_{i=1}^d (H^{i}\cdot X^{\e,i})_s \right| >a \right) = -\infty.
  \end{equation*}
\end{definition}

It is easy to deduce that $\{X^\e\}_{\e>0}$ is UET if and only if for each $i$, $\{X^{\e,i}\}_{\e>0}$ is UET, if and only if $\{X^\e-X^\e_0\}_{\e>0}$ is UET. Besides, the UET property is preserved under additive operation. Due to the fact that if $H$ is an adapted process, then $H_-$ is predictable, the definition of UET given here is equivalent to the UET condition for one-dimensional semimartingales proposed in \cite[Definition 1.1]{Gar08}. In view of these observations, we only focus on \emph{one-dimensional semimartingales starting at $0$} in the rest of this section.


The following lemma is taken from \cite[Lemma 2.5]{Gar08}, which provides some sufficient conditions for the UET of continuous local martingales and processes with locally finite variation.
\begin{lemma}\label{exp-tight-bv-conti}
  (i). Let $\{B^\e\}_{\e>0}$ be a family of one-dimensional c\`adl\`ag processes with locally finite variation. If the family $\{V(B^\e)_t\}_{\e>0}$ is exponentially tight for each $t>0$, then $\{B^\e\}_{\e>0}$ is UET.

  (ii). Let $\{M^\e\}_{\e>0}$ be a family of one-dimensional continuous local martingales. If the family $\{\frac{1}{\e}\langle M^\e,M^\e\rangle_t\}_{\e>0}$ is exponentially tight for each $t>0$, then $\{M^\e\}_{\e>0}$ is UET.
\end{lemma}

We can prove a similar statement as Lemma \ref{exp-tight-bv-conti}.(ii) for the purely discontinuous local martingales with uniformly bounded jumps, by utilizing the exponential estimate in Lemma \ref{pure-disc}.

\begin{lemma}\label{exp-tight-pure-disc}
  Let $\{M^\e\}_{\e>0}$ be a family of one-dimensional c\`adl\`ag purely discontinuous local martingales starting at $0$. 
  Let $\nu^\e$ be the L\'evy system of each $M^\e$. Assume for each $t>0$, the family $\{\sup_{0\le s\le t}|\Delta M^\e_s|\}_{\e>0}$ is exponentially tight and for some $r>0$
  \begin{equation}\label{cond-1}
    \lim_{A\to\infty}\lim_{\eta\to\infty}\limsup_{\e\to0}\e\log \P^\e\left( \e e^{|x|^2/(\e r)^2}\ind_{\{|x|\le A\}}* \nu^\e_t \ge \eta \right) = -\infty.
  \end{equation}
  Then the family $\{M^\e\}_{\e>0}$ is UET.
\end{lemma}
\begin{proof}
  Denote by $\mu^\e$ the jump measure of each $M^\e$. Let $H\in \Pred^\e_1$. By \cite[Corollary I.4.55.(d) and Eq. (I.4.36)]{JS13}, each $H\cdot M^\e$ is a purely discontinuous local martingale with jumps $|\Delta(H\cdot M^\e)| = |H\Delta M^\e| \le |\Delta M^\e|$. Denote by $\hat\mu^\e$ the jump measure associated to each $H\cdot M^\e$. Using Lemma \ref{pure-disc} and Lemma \ref{estimates}, we have
  \begin{equation}\label{est-17}
    \begin{split}
      &\ \P^\e\left( \sup_{0\le s\le t}|(H\cdot M^\e)_s|\ge a \right) \\
      = &\ \P^\e\left( \sup_{0\le s\le t}|(H\cdot M^\e)_s|\ge a, |x|^2 * \hat\mu^\e_t <b\e, \sup_{0\le s\le t}|\Delta M^\e_s|\le A \right) \\
      &\ + \P^\e\left( \sup_{0\le s\le t}|(H\cdot M^\e)_s|\ge a, |x|^2 * \hat\mu^\e_t \ge b \e, \sup_{0\le s\le t}|\Delta M^\e_s|\le A \right) \\
      &\ + \P^\e\left( \sup_{0\le s\le t}|(H\cdot M^\e)_s|\ge a, \sup_{0\le s\le t}|\Delta M^\e_s|> A \right) \\
      \le &\ 2\exp\left(-\frac{a^2}{4cb\e}\right) + \P^\e\left( \frac{1}{\e}(|Hx|^2 \ind_{\{|x|\le A\}}) * \mu^\e_t \ge b \right) + \P^\e\left( \sup_{0\le s\le t}|\Delta M^\e_s|> A \right) \\
      \le &\ 2\exp\left(-\frac{a^2}{4cb\e}\right) + \P^\e\left( \frac{1}{\e^2 r^2}(|x|^2 \ind_{\{|x|\le A\}}) * \mu^\e_t \ge \frac{b}{\e r^2} \right) + \P^\e\left( \sup_{0\le s\le t}|\Delta M^\e_s|> A \right) \\
      \le &\ 2\exp\left(-\frac{a^2}{4cb\e}\right) + \exp\left( \frac{\eta-b/r^2}{\e} \right) + \P^\e\left( \e e^{|x|^2/(\e r)^2}\ind_{\{|x|\le A\}}* \nu^\e_t \ge \eta \right) + \P^\e\left( \sup_{0\le s\le t}|\Delta M^\e_s|> A \right),
    \end{split}
  \end{equation}
  where $c>0$ is some constant. Therefore,
  \begin{equation*}
    \begin{split}
      &\ \limsup_{\e\to0}\sup_{H\in \Pred^\e_1} \e\log \P^\e\left( \sup_{0\le s\le t}|(H\cdot M^\e)_s|\ge a \right) \\
      \le&\ \left(-\frac{a^2}{4cb}\right) \vee (\eta-b) \vee \limsup_{\e\to0}\e\log \P^\e\left( \e e^{|x|^2/(\e r)^2}\ind_{\{|x|\le A\}}* \nu^\e_t \ge \eta \right) \vee \limsup_{\e\to0}\e\log \P^\e\left( \sup_{0\le s\le t}|\Delta M^\e_s|> A \right).
    \end{split}
  \end{equation*}
  By letting $a\to\infty$, $b\to\infty$, $\eta\to\infty$ and $A\to\infty$ successively, the UET of $\{M^\e\}_{\e>0}$ follows from the exponential tightness of $\{\sup_{0\le s\le t}|\Delta M^\e_s|\}_{\e>0}$ and \eqref{cond-1}.
\end{proof}

\begin{remark}
  We can modify the proof to obtain anther criterion for the UET of $\{M^\e\}_{\e>0}$. Note that $[M^\e,M^\e]_t = \sum_{0\le s\le t}|\Delta M^\e_s|^2 = |x|^2 * \mu^\e_t$ since $M^\e$ is purely discontinuous (see \cite[Lemma I.4.51]{JS13}). From the second inequality sign of \eqref{est-17}, we have
  \begin{equation*}
    \begin{split}
      \P^\e\left( \sup_{0\le s\le t}|(H\cdot M^\e)_s|\ge a \right) \le &\ 2\exp\left(-\frac{a^2}{4cb\e}\right) + \P^\e\left( \frac{1}{\e}|x|^2 * \mu^\e_t \ge b \right) + \P^\e\left( \sup_{0\le s\le t}|\Delta M^\e_s|> A \right) \\
      = &\ 2\exp\left(-\frac{a^2}{4cb\e}\right) + \P^\e\left( \frac{1}{\e}[M^\e,M^\e]_t \ge b \right) + \P^\e\left( \sup_{0\le s\le t}|\Delta M^\e_s|> A \right).
    \end{split}
  \end{equation*}
  Hence, it is easy to conclude that if the families $\{\frac{1}{\e} [M^\e,M^\e]_t\}_{\e>0}$ and $\{\sup_{0\le s\le t}|\Delta M^\e_s|\}_{\e>0}$ are exponentially tight for each $t>0$, then $\{M^\e\}_{\e>0}$ is also UET.
\end{remark}


We will give an equivalent condition for the exponential tightness of $\{\sup_{0\le s\le t}|\Delta M^\e_s|\}_{\e>0}$ in the forthcoming Lemma \ref{exp-tight-large-jumps}. Moreover, we can see from the proof of Lemma \ref{exp-tight-pure-disc} that the assumption of the exponential tightness of $\{\sup_{0\le s\le t}|\Delta M^\e_s|\}_{\e>0}$ can be removed as soon as each $\Delta M^\e$ is bounded process. Hence, we have the following corollary.
\begin{corollary}\label{exp-tight-pure-disc-bdd-jumps}
  Let $\{M^\e\}_{\e>0}$ be a family of one-dimensional c\`adl\`ag purely discontinuous local martingales starting at $0$ satisfying $|\Delta M^\e| \le A^\e$ for all $\e>0$ with some constants $A^\e>0$.
  Let $\nu^\e$ be the L\'evy system of each $M^\e$. Assume for each $t>0$, the family $\{\e e^{|x|^2/\e^2}\ind_{\{|x|\le A^\e\}}* \nu^\e_t\}_{\e>0}$ is exponentially tight. Then the family $\{M^\e\}_{\e>0}$ is UET.
\end{corollary}

Let $\{X^\e\}_{\e>0}$ be a family of one-dimensional c\`adl\`ag semimartingales with canonical decomposition $X^\e = \check X^\e(h) + B^\e(h) + X^{\e,c} + M^{\e,d}(h)$ and characteristics $(B^\e(h),C^\e,\nu^\e)$ associated to a given truncation function $h$. In Lemma \ref{exp-tight-bv-conti} and Corollary \ref{exp-tight-pure-disc-bdd-jumps}, we have found the sufficient conditions for UET property of the last three terms in the decomposition. It is only left to seek the conditions for the UET property of $\{\check X^\e(h)\}_{\e>0}$ in terms of characteristics. Note that $\check X^\e(h)$ is locally of finite variation. Lemma \ref{exp-tight-bv-conti} is applicable and we need to give some conditions for the exponential tightness of $\{V(\check X^\e(h))_t\}_{\e>0}$ for each $t>0$.


\begin{lemma}\label{exp-tight-large-jumps}
  Let $t>0$ be fixed. With the previous notation, the family $\{V(\check X^\e(h))_t\}_{\e>0}$ is exponentially tight for all truncation function $h$ if and only if the following two properties hold: \\
  (i). The family $\{\sup_{0\le s\le t}|\Delta X^\e_s|\}_{\e>0}$ is exponentially tight. \\
  (ii). For all $r>0$, the family $\{\sum_{0\le s\le t} \ind_{\{|\Delta X^\e_s|>r\}}\}_{\e>0}$ is exponentially tight.
\end{lemma}

\begin{proof}
  Fix $t>0$. Recall that $h_r(x)=x\ind_{\{|x|\le r\}}$ is a truncation function for each $r>0$. Then $\check X^\e(h_r)_t = \sum_{0\le s\le t} \Delta X^\e_s \ind_{\{|\Delta X^\e_s|>r\}}$ and $V(\check X^\e(h_r))_t = \sum_{0\le s\le t} |\Delta X^\e_s| \ind_{\{|\Delta X^\e_s|>r\}}$. Assume first the family $\{V(\check X^\e(h))_t\}_{\e>0}$ is exponentially tight for all $h$. For $a>1$, we have
  \begin{equation*}
    \left\{ \sup_{0\le s\le t}|\Delta X^\e_s| >a \right\} \subset \left\{ V(\check X^\e(h_1))_t >a \right\},
  \end{equation*}
  and (i) follows. For $r>0$, $\sum_{0\le s\le t} \ind_{\{|\Delta X^\e_s|>r\}} \le \frac{1}{r} V(\check X^\e(h_r))_t$ which implies (ii). Now we assume (i) and (ii) to hold. For each fixed truncation function $h$, there exists $r>0$ such that $|h|\ge |h_r|$ and then $V(\check X^\e(h)) \le V(\check X^\e(h_r))$. Hence, it is enough to show that $\{V(\check X^\e(h_r))_t\}_{\e>0}$ is exponentially tight for all $r>0$. For all $b>r$, we have
  \begin{equation*}
    \left\{ \sup_{0\le s\le t}|\Delta X^\e_s| \le b \right\} \subset \left\{ V(\check X^\e(h_r))_t \le b \sum_{0\le s\le t} \ind_{\{|\Delta X^\e_s|>r\}} \right\}.
  \end{equation*}
  Hence, for each $a>0$,
  \begin{equation}\label{est-16}
    \begin{split}
      \P^\e\left( V(\check X^\e(h_r))_t > a \right) & = \P^\e\left( V(\check X^\e(h_r))_t > a, \sup_{0\le s\le t}|\Delta X^\e_s| \le b \right) + \P^\e\left( V(\check X^\e(h_r))_t > a, \sup_{0\le s\le t}|\Delta X^\e_s| > b \right) \\
         & \le \P^\e\left( \sum_{0\le s\le t} \ind_{\{|\Delta X^\e_s|>r\}} > \frac{a}{b} \right) +  \P^\e\left( \sup_{0\le s\le t}|\Delta X^\e_s| > b \right),
    \end{split}
  \end{equation}
  Then
  \begin{equation*}
    \begin{split}
      \limsup_{\e\to0} \e\log \P^\e\left( V(\check X^\e(h_r))_t > a \right) \le&\ \limsup_{\e\to0} \e\log \P^\e\left( \sum_{0\le s\le t} \ind_{\{|\Delta X^\e_s|>r\}} > \frac{a}{b} \right) \\
         &\ \vee \limsup_{\e\to0} \e\log \P^\e\left( \sup_{0\le s\le t}|\Delta X^\e_s| > b \right),
    \end{split}
  \end{equation*}
  from which the exponential tightness of $\{V(\check X^\e(h))_t\}_{\e>0}$ follows by letting first $a\to\infty$ and then $b\to\infty$.
\end{proof}

\begin{remark}
  Intuitively, condition (i) means that the probability of the process $X^\e$ possessing large jumps is exponentially small. Condition (ii) means the probability of the process $X^\e$ possessing large amount of large jumps is exponentially small.
\end{remark}

In general, for a family of semimartingales $\{X^\e\}_{\e>0}$, UET is not implied by exponential tightness, and does not imply exponential tightness either. Here is however a connection between these two notions.
\begin{lemma}\label{UET-exp-tight}
  The exponential tightness of $\{X^\e\}_{\e>0}$ implies condition (i) and (ii) in Lemma \ref{exp-tight-large-jumps}.
\end{lemma}
\begin{proof}
  The result follows from Lemma \ref{exp-tight} and the following observations
  \begin{align*}
    \left\{\sup_{0\le s\le t}|\Delta X^\e_s| \ge a\right\} &\subset \left\{\sup_{0\le s\le t}|X^\e_s| \ge a/2 \right\}, \\
    \left\{ \sum_{0\le s\le t} \ind_{\{|\Delta X^\e_s|>r\}} \ge a \right\} &\subset \left\{ w_T(X^\e,T/a) \ge r \right\}.
  \end{align*}
\end{proof}

To obtain more appropriate conditions, we take the truncation function $h$ to be the $\e$-dependent function $h_{\e r}(x)=x\ind_{\{|x|\le \e r\}}$ with some $r>0$, and consider the decomposition $X^\e = \check X^\e(h_{\e r}) + B^\e(h_{\e r}) + X^{\e,c} + M^{\e,d}(h_{\e r})$.

\begin{lemma}\label{exp-tight-large-jumps-2}
  If for each $t>0$, if condition (i) in Lemma \ref{exp-tight-large-jumps} and the following condition hold: \\
  (iii). $\lim_{a\to\infty} \lim_{b\to\infty} \limsup_{\e\to0} \e\log \P^\e ( \sum_{0\le s\le t} |\Delta X^\e_s|\ind_{\{\e r<|\Delta X^\e_s|\le b\}} > a ) = -\infty$. \\
  Then the family $\{V(\check X^\e(h_{\e r}))_t\}_{\e>0}$ is exponentially tight.
\end{lemma}

\begin{proof}
  The result follows by modifying the inequality \eqref{est-16} to
  \begin{equation*}
    \P^\e\left( V(\check X^\e(h_{\e r}))_t > a \right) \le \P^\e\left( \sum_{0\le s\le t} |\Delta X^\e_s|\ind_{\{\e r<|\Delta X^\e_s|\le b\}} > a \right) +  \P^\e\left( \sup_{0\le s\le t}|\Delta X^\e_s| > b \right).
  \end{equation*}
\end{proof}


The following lemma provides some conditions for the UET property of $\{B^\e(h_{\e r})\}_{\e>0}$.
\begin{lemma}\label{UET-B}
  If for each $t>0$, the families $\{|x|\ind_{\{\e r<|x|\le b_0\}}*\nu^\e_t\}_{\e>0}$ and $\{V(B^\e(h_{b_0}))_t\}_{\e>0}$ are exponentially tight for some $b_0>0$. Then the family $\{B^\e(h_{\e r})\}_{\e>0}$ is UET.
\end{lemma}
\begin{proof}
  Note that $B^\e(h_{\e r}) = B^\e(h_{b_0}) + (h_{\e r}-h_{b_0})*\nu^\e$. We have $V(B^\e(h_{\e r}))\le V(B^\e(h_{b_0})) + V((h_{\e r}-h_{b_0})*\nu^\e)$. By \cite[Proposition II.2.9]{JS13}, there exists a predictable locally integrable increasing process $R^\e$ and a predictable transition kernel $K^\e_t(dx)$ from $\Omega\times\R_+$ into $\R^d$, such that $\nu^\e(dx,dt) = K^\e_t(dx)dR^\e_t$. Hence
  \begin{equation*}
    \begin{split}
      &\ V((h_{\e r}-h_{b_0})*\nu^\e) = \int_0^t \left| \int_{\R^d} (h_{\e r}(x)-h_{b_0}(x)) K^\e_s(dx) \right| dV(R^\e)_s \\
      \le&\ \int_0^t \int_{\e r<|x|\le b_0} |x| K^\e_s(dx) dV(R^\e)_s = |x|\ind_{\{\e r<|x|\le b_0\}}*\nu^\e_t.
    \end{split}
  \end{equation*}
  The result follows from Lemma \ref{exp-tight-bv-conti}.(i).
\end{proof}

\begin{remark}
  In fact, the proof of previous lemma yields a general conclusion. If for each $t>0$, the families $\{|x|*\nu^\e_t\}_{\e>0}$ and $\{V(B^\e(h_{b_0}))_t\}_{\e>0}$ are exponentially tight for some $b_0>0$. Then the family $\{B^\e(h_b)\}_{\e>0}$ is UET for all $b>0$.
\end{remark}

We are in position to give the sufficient conditions for the UET property of $\{X^\e\}_{\e>0}$.
\begin{proposition}[Criterion for UET property]\label{exp-tight-general}
  Let $\{X^\e\}_{\e>0}$ be a family of one-dimensional c\`adl\`ag semimartingales with characteristics $(B^\e(h),C^\e,\nu^\e)$ associated to a truncation function $h$. Suppose (i) in Lemma \ref{exp-tight-large-jumps} holds. Assume there exist $b_0>0$ and $r>0$ such that for each $t>0$, the families $\{V(B^\e(h_{b_0}))_t\}_{\e>0}$ and $\{\frac{1}{\e} C^\e_t\}_{\e>0}$ are exponentially tight, and the following holds,
  \begin{equation}\label{cond-2}
    \lim_{\eta\to\infty} \lim_{b\to\infty} \limsup_{\e\to0} \e\log \P^\e \left( \e\exp\left(\frac{|x|}{\e r}\vee 1 \right)\ind_{\{0<|x|\le b\}} * \nu^\e_t >\eta \right) = -\infty.
  \end{equation}
  Then $\{X^\e\}_{\e>0}$ is UET.
\end{proposition}
\begin{proof}
  We use the decomposition $X^\e = \check X^\e(h_{\e r}) + B^\e(h_{\e r}) + X^{\e,c} + M^{\e,d}(h_{\e r})$. The UET property of $\{X^{\e,c}\}$ and $\{\check X^\e(h_{\e r})\}$ follow from Lemma \ref{exp-tight-bv-conti}.(ii) and Lemma \ref{exp-tight-large-jumps-2}. The UET property of $\{B^\e(h_{\e r})\}$ follows from Lemma \ref{UET-B}, since the exponential tightness of $\{|x|\ind_{\{\e r<|x|\le b\}}*\nu^\e_t\}_{\e>0}$ is a consequence of \eqref{cond-2} and the observation $|x|\ind_{\{\e r<|x|\le b_0\}} \le \e\exp(|x|/(\e r))\ind_{\{\e r<|x|\le b\}}$ for all $b\ge b_0$. Finally, to check the UET property of $\{M^{\e,d}(h_{\e r})\}$, we use Corollary \ref{exp-tight-pure-disc-bdd-jumps} by taking $A^\e = \e r$. The proof is complete.
\end{proof}

\section{Proofs of main results}\label{sec-4}

For each $\e>0$, we have a filtered probability space $(\Omega^\e,\F^\e,\{\F^\e_t\}_{t\ge0},\P^\e)$ endowed with a $d$-dimensional semimartingale $X^\e$ and an $n$-dimensional c\`adl\`ag adapted process $U^\e$. We also have a bounded Lipschitz function $F: \R^n \to \R^{n\times d}$, so that each stochastic differential equation
\begin{equation*}
  Y^\e = U^\e + F(Y^\e_-) \cdot X^\e,
\end{equation*}
has a unique global solution $Y^\e$ which is an $n$-dimensional process. Suppose the function $F$ and its Lipschitz constant are both bounded by $c>0$, that is, $|F(y_1) - F(y_2)| \le c|y_1-y_2|$ and $|F(y)|\le c$ for all $y_1,y_2,y$. 

The following Gronwall-type inequality is adapted from \cite[Lemma IX.6.3]{JS13}. The proof is almost the same and shall be omitted.
\begin{lemma}\label{Gronwall}
  Let $A$ be a nondecreasing c\`adl\`ag process, $H$ be a nonnegative adapted process, such that $\E( (H_-\cdot A)_\infty ) < \infty$ and $A_\infty \le K$ identically for some constant $K$.
  Suppose that for each stopping time $T$ we have
  \begin{equation*}
    \E(H_{T-}) \le \alpha + \E( (H_- \cdot A)_{T-} ),
  \end{equation*}
  for some constant $\alpha$. Then $\E(H_\infty) \le \alpha e^{Kt}$.
\end{lemma}

\begin{lemma}\label{est-8}
  For each $\e>0$, let $X^\e$ be a special semimartingale with canonical decomposition $X^\e = B^\e + M^\e$, let $F^\e$ be an $\{\F^\e_t\}_{t\ge0}$-adapted processes and $U^\e$ be a c\`adl\`ag processes. Suppose $|\Delta X^\e| \le A$ and $|U^\e| \le z_0$ for all $\e>0$, with some constants $A>0$ and $z_0>0$. Let $T^\e$ be an $\{\F^\e_t\}$-stopping time. Let $Z^\e = U^\e + F^\e_- \cdot X^\e$. Suppose there exist positive constants $c, K$ and $\rho$, such that for every $t\in [0,T^\e)$ and all $\e>0$,
  \begin{equation}\label{bound}
    |F^\e_t| \le c\left((\rho^2 + |Z^\e_t|^2)^{1/2} \wedge 1 \right),
  \end{equation}
  and
  \begin{equation}\label{G}
    \Xi(X^\e)_t := |V(B^\e)|_t + \frac{1}{\e}|\langle M^{\e,c}, M^{\e,c} \rangle|_t + \frac{1}{\e} \int_{\Ro d} e^{2 c|x|/\e} |x|^2 \nu^\e( dx,[0,t]) \le K.
  \end{equation}
  Then for all $a > 3z_0 +c A$ and $0<\e \le 1$,
  \begin{equation*}
    \e \log \P^\e\left( \sup_{t\in[0,T^\e]} |Z^\e_t|\ge a \right) \le c_1+ \log\left( \frac{\rho^2}{\rho^2+(a - 3z_0 - c A)^2}\right),
  \end{equation*}
  where $c_1 = (2c+4c^2)K$.
\end{lemma}
\begin{proof}
  Define $\phi^\e(z) = (\rho^2 + |z|^2)^{1/\e}$. Then
  \begin{equation*}
    \partial_i \phi^\e(z) = \frac{2\phi^\e(z)}{\e(\rho^2+|z|^2)}z_i, \quad \partial_i\partial_j \phi^\e(z) = \frac{2\phi^\e(z)}{\e(\rho^2+|z|^2)} \left( \delta_{ij} + 2\left( \frac{1}{\e}-1 \right) \frac{z_i z_j}{\rho^2+|z|^2}\right).
  \end{equation*}
  Then for small $\e>0$,
  \begin{equation}\label{est-7}
    |D\phi^\e(z)| \le \frac{2|z|}{\e(\rho^2+|z|^2)}\phi^\e(z),\quad |D^2\phi^\e(z)| \le \frac{4}{\e^2(\rho^2+|z|^2)} \phi^\e(z).
  \end{equation}
  Denote $G^\e:=Z^\e - U^\e = F^\e_- \cdot X^\e$. Let $\Phi^\e := \phi^\e(G^\e)$. By It\^o's formula (see, e.g., \cite[Theorem I.4.57]{JS13}),
  \begin{equation*}
    \begin{split}
       \Phi^\e_t =&\ \phi^\e(0) + \partial_i\phi^\e(G^\e_-)F_{j,-}^{\e,i} \cdot B^{\e,j} + \partial_i\phi^\e(G^\e)F^{\e,i}_{j} \cdot M^{\e,j,c} + \partial_i\partial_j\phi^\e(G^\e) F_{k}^{\e,i} F_l^{\e,j} \cdot \langle M^{\e,k,c}, M^{\e,l,c} \rangle \\
         &\ + \left( \phi^\e(G^\e_- + F^\e_- x) - \phi^\e(G^\e_-) \right) * (\mu^\e-\nu^\e) \\
         &\ + \left( \phi^\e(G^\e_- + F^\e_- x) - \phi^\e(G^\e_-) - \partial_i \phi(G^\e_-)F^{\e,i}_{j,-} x^j \right) * \nu^\e.
    \end{split}
  \end{equation*}
  Define a stopping time $T^{\e,a} = \inf\{t\ge0: |Z^\e_t|\ge a \} \wedge T^\e$. When $t\in[0,T^{\e,a})$, we have $\sup_{0\le s\le t}|Z^\e_s| \le a$, then the bound \eqref{bound} and the assumption that $|U^\e|\le z_0$ yield $|D\phi^\e(G^\e_{t})||F^\e_t|\le c_2(\e)$ with $c_2(\e)>0$. Moreover, using Taylor's theorem, there exists $c_3(\e)>0$ such that for all $|x|\le A$ and $t\in[0,T^{\e,a})$,
  \begin{equation*}
    \left| \phi^\e(G^\e_{t-} + F^\e_{t-} x) - \phi^\e(G^\e_{t-}) \right| \le |D\phi^\e(G^\e_{t-}+ \theta F^\e_{t-} x)||F^\e_t||x| \le c_3(\e)|x|,
  \end{equation*}
  where $\theta$ is a $(0,1)$-valued random variable. Then the bound \eqref{G} yields
  \begin{gather*}
    \E^\e \left( \int_0^{T^{\e,a}-} |D\phi^\e(G^\e_{t})|^2|F^\e_{t}|^2 d|\langle M^{\e,c}, M^{\e,c} \rangle|_t\right) < \infty, \\
    \E^\e \left( \int_0^{T^{\e,a}-}\int_{\Ro d} \left| \phi^\e(G^\e_{t-} + F^\e_{t-} x) - \phi^\e(G^\e_{t-}) \right|^2 \nu^\e(dx,dt) \right) < \infty,
  \end{gather*}
  where in the second inequality we used the fact that $\nu^\e$ is support on $\{(x,t):|x|\le A\}$, since $|\Delta X^\e| \le A$. Thus, the stochastic integrals $\partial_i\phi^\e(G^\e)F^{\e,i}_{j} \cdot M^{\e,j,c}$ and $( \phi^\e(G^\e_- + F^\e_- x) - \phi^\e(G^\e_-) ) * (\mu^\e-\nu^\e)$ are all martingales up to $T^{\e,a}-$. This yields for each $\{\F^\e_t\}$-stopping time $S^\e$,
  \begin{equation*}
    \begin{split}
       &\ \E^\e\left( \Phi^\e_{(S^\e \wedge T^{\e,a})-} \right) \\
       =&\ \phi^\e(0) + \E^\e\left( \int_0^{(S^\e \wedge T^{\e,a})-} \partial_i\phi^\e(G^\e_s)F_{\e,j,s}^i dB^{\e,j}_s \right) \\
       &\ + \E^\e\left( \int_0^{(S^\e \wedge T^{\e,a})-} \partial_i\partial_j\phi^\e(G^\e_s) F_{\e,k,s}^i F_{\e,l,s}^j d \langle M^{\e,k,c}, M^{\e,l,c} \rangle_s \right) \\
         &\ + \frac{1}{2}\E^\e\left( \int_0^{(S^\e \wedge T^{\e,a})-} \int_{\Ro d}\left( \phi^\e(G^\e_{s-} + F^\e_{s-} x) - \phi^\e(G^\e_{s-}) - \partial_i \phi^\e(G^\e_{s-})F^{\e,i}_{j,s-} x^j \right) \nu^\e(dx,ds) \right) \\
      =:&\ \phi^\e(0) + I^\e_1 +I^\e_2 +I^\e_3.
    \end{split}
  \end{equation*}
  By \eqref{bound} and \eqref{est-7}, it is easy to get
  \begin{equation*}
    |I^\e_1| \le \frac{2c}{\e} \E^\e \int_0^{(S^\e \wedge T^{\e,a})-} \Phi^\e_s d|V(B^\e)|_s, \quad |I^\e_2| \le \frac{4c^2}{\e^2} \E^\e \int_0^{(S^\e \wedge T^{\e,a})-} \Phi^\e_s d|\langle M^{\e,c}, M^{\e,c} \rangle|_s.
  \end{equation*}
  For $I^\e_3$, using Taylor's theorem, we have for some $(0,1)$-valued random variable $\theta$,
  \begin{equation*}
    \begin{split}
       |I^\e_3| &\le \E^\e \int_0^{(S^\e \wedge T^{\e,a})-} \int_{\Ro d} |D^2 \phi^\e(G^\e_{s-} + \theta F^\e_{s-} x)| |F^\e_{s-}|^2 |x|^2 \nu^\e(dx,ds) \\
         &\le \frac{4c^2}{\e^2}\E^\e \int_0^{(S^\e \wedge T^{\e,a})-} \int_{\Ro d} (\rho^2 + |G^\e_{s-} + \theta F^\e_{s-} x|^2)^{\frac{1}{\e}-1}(\rho^2 + |G^\e_{s-}|^2) |x|^2 \nu^\e(dx,ds) \\
         &\le \frac{4c^2}{\e^2}\E^\e \int_0^{(S^\e \wedge T^{\e,a})-} \int_{\Ro d} \left( 1+ \frac{|G^\e_{s-}||F^\e_{s-}x|}{\rho^2 + |G^\e_{s-}|^2} + \frac{|F^\e_{s-}x|^2}{\rho^2 + |G^\e_{s-}|^2} \right)^{\frac{1}{\e}-1}\Phi^\e_{s-} |x|^2 \nu^\e(dx,ds) \\
         &\le \frac{4c^2}{\e^2}\E^\e \int_0^{(S^\e \wedge T^{\e,a})-}\Phi^\e_{s-} \int_{\Ro d} (1+c|x|)^{\frac{2}{\e}-2} |x|^2 \nu^\e(dx,ds).
    \end{split}
  \end{equation*}
  A change of variable $x=\e y$ yields
  \begin{equation*}
    \begin{split}
      &\ \int_{\Ro d} (1+c|x|)^{\frac{2}{\e}-2} |x|^2 \nu^\e(dx,ds) = \e^2\int_{\Ro d} (1+\e c|y|)^{\frac{2}{\e}-2} |y|^2 \nu^\e( d(\e y),ds) \\
      \le&\ \e^2\int_{\Ro d} e^{2c|y|} |y|^2 \nu^\e( d(\e y),ds) = \int_{\Ro d} e^{2 c|x|/\e} |x|^2 \nu^\e( dx,ds).
    \end{split}
  \end{equation*}
  Therefore,
  \begin{equation*}
    \E^\e\left( \Phi^\e_{(S^\e \wedge T^{\e,a})-} \right) \le \phi^\e(0) + (2c+4c^2) \frac{1}{\e} \E^\e \int_0^{S^\e -} \Phi^\e_{(s\wedge T^{\e,a})-} d\Xi(X^\e)_s.
  \end{equation*}
  Since $\Xi(X^\e) \le K$ uniformly, by virtue of the Gronwall-type inequality in Lemma \ref{Gronwall}, we have
  \begin{equation*}
    \E^\e(\Phi^\e_{T^{\e,a}-}) \le \phi^\e(0)e^{c_1/\e}.
  \end{equation*}
  Note that $|\Delta Z^\e| \le |\Delta U^\e| + |F^\e_-||\Delta X^\e| \le 2z_0 + c A$. Let $b = a - (3z_0 +c A)$. Then by Chebycheff's inequality, for each $a > 3z_0 +c A$,
  \begin{equation*}
    \begin{split}
       \e \log \P^\e\left(\sup_{0\le t\le T^\e} |Z^\e_t|> a \right) & = \e \log \P^\e(|Z^\e_{T^{\e,a}}| \ge a) \le \e \log \P^\e\left(|Z^\e_{T^{\e,a}-}| \ge a - (2z_0 +c A) \right) \\
       & \le \e \log \P^\e\left(|G^\e_{T^{\e,a}-}| \ge b \right) = \e \log \P^\e(\Phi^\e_{T^{\e,a}-} \ge \phi^\e(b)) \\
         & \le c_1 + \e \log \phi^\e(0) - \e \log \phi^\e(b) = c_1 +\log\left( \frac{\rho^2}{\rho^2+b^2}\right).
    \end{split}
  \end{equation*}
  The proof is completed.
\end{proof}

\begin{lemma}\label{LDP-integral}
  Let $\{F^\e\}_{\e>0}$ be a family of $\{\F^\e_t\}_{t\ge0}$-adapted processes. Assume the family $\{X^\e\}_{\e>0}$ is UET. If the family $\{(X^\e,U^\e,F^\e)\}_{\e>0}$ satisfies the LDP with good rate function $I^\sharp$, then the family $\{(X^\e,U^\e,F^\e,F_-^\e\cdot X^\e)\}_{\e>0}$ also satisfies the LDP with the following good rate function:
  \begin{equation*}
    I(x,u,f,w) =
    \begin{cases}
      I^\sharp(x,u,f), & w = f\cdot x \text{ and } x \text{ is locally of finite variation}, \\
      \infty, & \text{otherwise}.
    \end{cases}
  \end{equation*}
\end{lemma}
\begin{proof}
   Observe that $F_-^\e\cdot X^\e = (U_-^\e,F_-^\e) \cdot (0,X^\e)^T$. Then the result is a corollary of \cite[Theorem 1.2]{Gar08}.
\end{proof}

Now we are in position to prove the main result Theorem \ref{LDP-special}. 

\begin{proof}[\textbf{Proof for Theorem \ref{LDP-special}}]

  Firstly, we note that the UET property of the family $\{X^\e\}$ is implied by the assumptions and Proposition \ref{exp-tight-general} as well as Lemma \ref{UET-exp-tight}, since $\{X^\e\}$ is exponentially tight. Our proof is divided into four steps.

  \emph{Step 1 (First localization).} We suppose $\{(X^\e,U^\e,Y^\e)\}$ is exponentially tight if in addition, the following condition is satisfied:
  \begin{condition}\label{Cond. 1}
    The processes $X^\e$ and $U^\e$ are uniformly bounded by a constant $K_1>0$.
  \end{condition}
  We deduce it holds in general. For each $p>0$, define $f_p:\R^{d} \to \R^{d}$ to be a $\C^2$ function with $|f_p|\le 2p$, $|Df_p|\vee |D^2f_p|\le 1$ and $f_p(x) = x$ for $|x|\le p$, also define $g_p:\R^{n} \to \R^{n}$ to be a $\C^2$ bounded function with $g_p(u) = u$ for $|u|\le p$. Let $X^{\e,p}: = f_p(X^\e)$ and $U^{\e,p}:=g_p(U^\e)$. Then for each $\e$, the following SDE has a unique solution
  \begin{equation*}
    Y^{\e,p} = U^{\e,p} +F(Y^{\e,p}_-) \cdot X^{\e,p}.
  \end{equation*}
  By the continuity and boundedness of $f_p$, each family $\{(X^{\e,p},U^{\e,p})\}_{\e>0}$ is also exponentially tight and uniformly bounded.
  Obviously, $|\Delta X^{\e,p}| = |f_p(X^\e) - f_p(X^\e_-)| \le 4p$, which implies that each $X^{\e,p}$ is a special semimartingale. Denote by $\mu^{\e,p}$ the jump measure associated to each $X^{\e,p}$ and by $\nu^{\e,p}$ its L\'evy system. Then for every $E\in\B(\Ro d)$,
  \begin{equation*}
    \mu^{\e,p}(E\times [0,t]) = \int_0^t\int_{\R^d} \ind_E(f_p(X_{s-}+x) - f_p(X_{s-})) \mu^\e(dx,ds).
  \end{equation*}
  This relation carries over to the L\'evy systems. Hence, $\nu^{\e,p}(\R^d\times\{t\}) =0$, which yields each $X^{\e,p}$ is quasi-left-continuous (see \cite[Corollary II.1.19]{JS13}). Let $(B^\e(h),C^\e,\nu^\e)$ be the characteristics of each $X^{\e,p}$ associated to truncation function $h$, let $X^{\e,p,c}$ be its continuous martingale part. 
  Using It\^o's formula, it is easy to deduce that
  \begin{align*}
    X^{\e,p,c} =&\ \partial_i f_p(X^\e_-) \cdot X^{\e,i,c}, \\
    B^{\e,p}(h) =&\ \partial_i f_p(X^\e_-)\cdot B^{\e,i}(h) + \frac{1}{2} \partial_i\partial_j f_p(X^\e_-) \cdot\langle X^{\e,i,c}, X^{\e,i,c}\rangle \\
    &\ + [h(f_p(X^\e_-+x) - f_p(X^\e_-)) - \partial_i f_p(X^\e_-) x^i] * \nu^\e.
  \end{align*}
  Since $|Df_p|\vee |D^2f_p|\le 1$, we have for $\e>0$ small enough,
  \begin{equation*}
    \begin{split}
      V(B^{\e,p}(h_b))_t \le&\ V(B^\e(h_b))_t + \frac{1}{2}\langle X^{\e,c},X^{\e,c}\rangle_t + \int_{|x|> b} |x| \nu^\e(dx,[0,t]) + \int_0^t\int_{\Ro d} |x|^2 \nu^\e(dx,ds) \\
      \le&\ V(B^\e(h_b))_t + \frac{1}{2}\langle X^{\e,c},X^{\e,c}\rangle_t + \left( 1+\frac{1}{b} \right)\int_{\Ro d} |x|^2 \nu^\e(dx,[0,t]) \\
      \le&\ V(B^\e(h_b))_t + \frac{1}{2}\langle X^{\e,c},X^{\e,c}\rangle_t + \e \int_{\Ro d} \exp(|x|/(\e r)) \nu^\e(dx,[0,t]),
    \end{split}
  \end{equation*}
  \begin{equation*}
    \frac{1}{\e}C^{\e,p} = \frac{1}{\e} \langle X^{\e,p,c},X^{\e,p,c}\rangle \le \frac{1}{\e}\langle X^{\e,c},X^{\e,c}\rangle = \frac{1}{\e} C^\e,
  \end{equation*}
  \begin{equation*}
    \begin{split}
      \e \int_{\Ro d} \exp\left(\frac{|x|}{\e r}\vee 1 \right) \nu^{\e,p}(dx,[0,t]) &= \e \int_{\Ro d} \exp\left(\frac{|f_p(X_{s-}+x) - f_p(X_{s-})|}{\e r}\vee 1 \right) \nu^\e(dx,[0,t]) \\
         & \le \e \int_{\Ro d} \exp\left(\frac{|x|}{\e r}\vee 1 \right) \nu^\e(dx,[0,t]).
    \end{split}
  \end{equation*}
  Hence, for each $t>0$, the three families in left hand side of the above inequalities are exponentially tight, by the exponential tightness of \eqref{three-families}. So our assumptions yield that the family $\{(X^{\e,p},U^{\e,p},Y^{\e,p})\}$ is exponentially tight, for all $p>0$.

  For each $\e, p>0$, define a stopping time
  $$T^{\e,p}:=\inf\{t\ge0: |X^\e_t|+|U^\e_t| \ge p\}.$$
  Then $T^{\e,p}$ is nondecreasing in $p$, and $(X^{\e,p},U^{\e,p}) = (X^{\e},U^{\e})$ on the interval $[0,T^{\e,p})$. By virtue of the exponential tightness of $(X^{\e},U^{\e})$ and Lemma \ref{exp-tight}, we know that for every $T>0$ and $M>0$, there exists $p_0>0$, such that
  \begin{equation*}
    \limsup_{\e\to0}\e\log \P^\e( T^{\e,p_0} \le T ) = \limsup_{\e\to0}\e\log \P^\e\left(\sup_{0\le t\le T}(|X^\e_t|+|U^\e_t|) \ge p_0 \right) \le -M.
  \end{equation*}
  Hence
  \begin{equation}\label{exp-11}
    \begin{split}
       &\ \P^\e\left( \sup_{0\le t\le T}(|X^{\e}_t| + |U^{\e}_t| + |Y^{\e}_t|)\ge a \right) \\
       =&\ \P^\e\left( \sup_{0\le t\le T}(|U^{\e}_t| + |Y^{\e}_t| + |X^{\e}_t|)\ge a, T^{\e,p_0}>T \right) + \P^\e\left( \sup_{0\le t\le T}(|U^{\e}_t| + |Y^{\e}_t| + |X^{\e}_t|)\ge a, T^{\e,p_0}\le T \right) \\
       \le&\ \P^\e\left( \sup_{0\le t\le T}(|U^{\e,p_0}_t| + |Y^{\e,p_0}_t| + |X^{\e,p_0}_t|)\ge a \right) + \P^\e\left( T^{\e,p_0}\le T \right),
    \end{split}
  \end{equation}
  and
  \begin{equation}\label{exp-12}
    \P^\e\left( w_T((U^{\e},Y^{\e},X^{\e}),\rho) \ge\eta \right) \le \P^\e\left( w_T((U^{\e,p_0},Y^{\e,p_0},X^{\e,p_0}),\rho) \ge\eta \right) + \P^\e\left( T^{\e,p_0}\le T \right).
  \end{equation}
  Therefore, the exponential tightness of $\{(X^{\e},U^{\e},Y^{\e})\}$ follows from that of $\{(X^{\e,p_0},U^{\e,p_0},Y^{\e,p_0})\}$ and Lemma \ref{exp-tight}.

  \emph{Step 2 (Second localization).} Suppose $\{(X^\e,U^\e,Y^\e)\}$ is exponentially tight if in addition, Cond. \ref{Cond. 1} and the following condition are both satisfied:
  \begin{condition}\label{Cond. 2}
    The increasing processes $\Xi(X^\e)$ associated to $X^\e$ in \eqref{G} are also uniformly bounded by a constant $K_2>0$.
  \end{condition}
  We deduce it still holds when only Cond. \ref{Cond. 1} is satisfied. For each $\e, p>0$, define a stopping time
  $$T^{\e,p}:=\inf\{t\ge0: \Xi(X^\e)_t \ge p\}.$$
  Let $X^{\e,p} := X^\e_{\cdot\wedge T^{\e,p} }$, $U^{\e,p} := U^\e_{\cdot\wedge T^{\e,p} }$, and each $Y^{\e,p}$ be the solution of the SDE
  \begin{equation*}
    Y^{\e,p} = U^{\e,p} +F(Y^{\e,p}_-) \cdot X^{\e,p},
  \end{equation*}
  Then obviously each $X^{\e,p}$ is a special semimartingale, and $Y^{\e,p} = Y^\e_{\cdot\wedge T^{\e,p} }$. It is obvious that (cf. \cite[Eq. (VI.1.9)]{JS13}) for all $T>0$,
  \begin{align*}
    \sup_{0\le t\le T}(|X^{\e,p}_t| + |U^{\e,p}_t|) &= \sup_{0\le t\le T \wedge T^{\e,p}}(|X^{\e}_t| + |U^{\e}_t|) \le \sup_{0\le t\le T }(|X^{\e}_t| + |U^{\e}_t|), \\
    w_T((X^{\e,p},U^{\e,p}),\rho) &= w_{T\wedge T^{\e,p}}((X^{\e},U^{\e}),\rho) \le w_{T}((X^{\e},U^{\e}),\rho).
  \end{align*}
  Since $\{(X^\e,U^\e)\}$ is exponentially tight, Lemma \ref{exp-tight} yields that $\{(X^{\e,p},U^{\e,p})\}$ is also exponentially tight, for each $p>0$. Let $X^{\e,p} = M^{\e,p} + B^{\e,p}$ be the canonical decomposition of each $X^{\e,p}$, and let $\nu^{\e,p}$ be the L\'evy system associated to each $X^{\e,p}$. Then for all $t>0$ and $E\in\B(\Ro d)$,
  \begin{align*}
    V(B^{\e,p}(h_b))_t & = V(B^\e(h_b))_{t\wedge T^{\e,p}} \le V(B^\e(h_b))_{t}, \\
    \langle X^{\e,p,c},X^{\e,p,c}\rangle_t & = \langle X^{\e,c},X^{\e,c}\rangle_{t\wedge T^{\e,p}} \le \langle X^{\e,c},X^{\e,c}\rangle_{t}, \\
    \nu^{\e,p}(E\times [0,t]) & = \nu^{\e}(E\times [0,t\wedge T^{\e,p}]) \le \nu^{\e}(E\times [0,t]).
  \end{align*}
  Using these one can easily deduce that the three families associated to $X^{\e,p}$ as in \eqref{three-families} is also exponentially tight. Moreover,
  \begin{equation*}
    \nu^{\e,p}(\R^d\times\{t\}) = \nu^{\e}(\R^d\times (\{t\} \cap [0,T^{\e,p}])) \le \nu^\e(\R^d\times\{t\}) = 0,
  \end{equation*}
  which leads to the quasi-left-continuity of $X^{\e,p}$, and thus $\Delta B^{\e,p} \equiv 0$ by \cite[Proposition II.2.29]{JS13}. Let $\Xi(X^{\e,p})$ be the increasing process associated to $X^{\e,p}$ as in \eqref{G}. Then $\Delta \Xi(X^{\e,p}) \equiv 0$. Hence, $\Xi(X^{\e,p})$ is uniformly bounded by $p$, for each $p>0$. Therefore, our assumptions yield each family $\{(X^{\e,p},U^{\e,p},Y^{\e,p})\}_{\e>0}$ is exponentially tight.

  Since the three family in \eqref{three-families} is exponentially tight, and since $0<r<\frac{1}{4c}$, we have for $\e>0$ small enough,
  \begin{equation*}
    \begin{split}
      &\ \frac{1}{\e} \int_{\Ro d} e^{2 c|x|/\e} |x|^2 \nu^\e( dx,[0,t]) = \e\int_{\Ro d} e^{2c|y|} |y|^2 \nu^\e( d(\e y),ds) \\
      \le&\ \e\int_{\Ro d} e^{(|y|/r)\vee1} \nu^\e( d(\e y),ds) = \e \int_{\Ro d} e^{(|x|/(\e r))\vee1} \nu^\e(dx,[0,t]).
    \end{split}
  \end{equation*}
  Thus, the family $\{\Xi(X^\e)_t\}_{\e>0}$ is exponentially tight for each $t>0$. Using the fact that each $\{\Xi(X^\e)\}$ is increasing, we have for every $T>0$ and $M>0$, there exists $p_0>0$, such that
  \begin{equation*}
    \limsup_{\e\to0}\e\log \P^\e( T^{\e,p_0} \le T ) = \limsup_{\e\to0}\e\log \P^\e\left( \Xi(X^\e)_T \ge p_0 \right) \le -M.
  \end{equation*}
  Then a similar argument as \eqref{exp-11} and \eqref{exp-12} yields that the family $\{(X^\e,U^{\e},Y^{\e})\}$ is exponentially tight.

  \emph{Step 3 (Exponential tightness of $\{(X^\e,U^\e,Y^\e)\}$).} In this step, we will assume Cond. \ref{Cond. 1} and Cond. \ref{Cond. 2} to hold, and prove the exponential tightness of $\{(X^\e,U^\e,Y^\e)\}$.

  Since $(X^\e,U^\e)$ obeys the LDP with good rate function, by Proposition \ref{Slominski}.(ii), we can associated with $(X^\e,U^\e)$ positive constants $\rho_p^N$ and $a(\e,p,i)\in (\frac{1}{2p},\frac{1}{p}]$, and stopping times $T^{\e,p}_i$ defined in \eqref{random-times}, with $(X^\e,U^\e)$ in place of $X^\e$, which satisfy \eqref{exp-tight-1} and \eqref{exp-tight-2}. To prove the exponential tightness of $\{(X^\e,U^\e,Y^\e)\}$, by Proposition \ref{Slominski}.(i), it is enough to show that the family $\{Y^\e\}$ satisfies Lemma \ref{exp-tight}.(i) as well as \eqref{exp-tight-2} with the same stopping times $T^{\e,p}_i$.


  Since $F$ is bounded Lipschitz and $X^\e$, $U^\e$ are uniformly bounded by $K_1$, we have $|\Delta X^\e|\le 2K_1$ and  $|F(Y^\e)|^2 \le c(1+|Y^\e|^2)$. By Lemma \ref{est-8}, for all $T>0$ and every $a > (3+2c)K_1$, $0<\e \ll 1$,
  \begin{equation}\label{est-9}
    \e \log \P^\e\left( \sup_{t\in[0,T]} |Y^\e_t|\ge a \right) \le c_1+ \log\left( \frac{1}{1+(a - (3+2c)K_1)^2}\right),
  \end{equation}
  with $c_1 = (2c+4c^2)K_2$. This implies that Lemma \ref{exp-tight}.(i) holds for $\{Y^\e\}$.

  To prove \eqref{exp-tight-2} for $\{Y^\e\}$, we first define
  \begin{equation*}
    X^{\e,p}_t = X^\e_{T^{\e,p}_i}, Y^{\e,p}_t = Y^\e_{T^{\e,p}_i}, U^{\e,p}_t = U^\e_{T^{\e,p}_i}, \quad\text{when } T^{\e,p}_i \le t < T^{\e,p}_{i+1}.
  \end{equation*}
  If $T^{\e,p}_i \le t< T^{\e,p}_{i+1}$, we have
  \begin{equation*}
    Y^\e_t - Y^{\e,p}_t = U^\e_t - U^{\e,p}_t + \int_{T^{\e,p}_i}^t F(Y^\e_{s-}) dX^\e_s.
  \end{equation*}
  Using the boundedness of $F$, for all $t$,
  \begin{equation*}
      |Y^\e_t - Y^{\e,p}_t| \le |U^\e_t - U^{\e,p}_t| + c |X^\e_t - X^{\e,p}_t| \le \frac{1+c}{p}.
  \end{equation*}
  Set $S^{\e,r} := \inf\{t\ge0: |Y^\e_t|\ge r\}$. Hence, for every $\rho, \eta>0$ and $0<\e \le 1$,
  \begin{equation}\label{est-10}
    \left[ \P^\e\left( \sup_{t\in[0,S^{\e,r}]} |Y^\e_t-Y^{\e,p}_t|\ge \eta \right) \right]^\e \le \ind_{\left\{\frac{1+c}{p}\ge \eta \right\}}.
  \end{equation}
  Note that for all $N>0$,
  \begin{equation*}
    \{S^{\e,r}\ge N\} \subset \left\{ \sup_{i} w\left(Y^\e,[T^{\e,p}_i,T^{\e,p}_{i+1}) \cap [0,N]\right) \le \sup_{t\in[0,S^{\e,r}]} |Y^\e_t - Y^{\e,p}_t| \right\},
  \end{equation*}
  and then
  \begin{equation*}
    \left\{ \sup_{i} w\left(Y^\e,[T^{\e,p}_i,T^{\e,p}_{i+1}) \cap [0,N]\right) \ge \eta \right\} \subset \left\{\sup_{t\in[0,S^{\e,r}]} |Y^\e_t - Y^{\e,p}_t| \ge \eta\right\} \cup \left\{\sup_{t\in[0,N]} |Y^\e_t|\ge r\right\}.
  \end{equation*}
  Combining with \eqref{est-9} and \eqref{est-10}, we have for $r$ large enough,
  \begin{equation*}
    \begin{split}
       \left[ \P^\e \left( \sup_{i} w\left(Y^\e,[T^{\e,p}_i,T^{\e,p}_{i+1}) \cap [0,N]\right) \ge \eta \right)\right]^\e \le \ind_{\left\{\frac{c}{p}\ge \eta \right\}} + \exp\left\{ c_1+ \log\left( \frac{1}{1+(r - (3+2c)K_1)^2}\right) \right\},
    \end{split}
  \end{equation*}
  By letting first $p\to\infty$ and then $r\to\infty$, we obtain \eqref{exp-tight-2} for $\{Y^\e\}$.

  \emph{Step 4 (Identification of the rate function).} In this final step, let the family $\{(X^\e,U^\e,Y^\e)\}_{\e>0}$ be exponentially tight. We show that for each subsequence $\{(X^{\e_k},U^{\e_k},Y^{\e_k})\}_{k=1}^\infty$, with $\e_k\to0$ as $k\to\infty$, which obeys an LDP, the rate function $I$ is given by \eqref{rate-func}. For notational simplicity, we still denote the subsequence $\e_k$ by $\e$.

  We follow the lines of \cite[Theorem 6.1]{Gan18}. By the contraction principle, the family $(X^\e,U^\e,F(Y^\e))$ obeys the LDP with good rate funtion $I^\sharp(x,u,f) = \inf\{I(x,u,y): f=F(y)\}$. Since $Y^\e = U^\e + F(Y^\e_-)\cdot X^\e$, Lemma \ref{LDP-integral} and the contraction principle yields that the family $(X^\e,U^\e,Y^\e,F(Y^\e))$ obeys the LDP with good rate function
  \begin{equation}\label{rate-J}
    \begin{split}
      J(x,u,y,f) & =
      \begin{cases}
        I^\sharp(x,u,f), & y = u+f\cdot x \text{ and } x \text{ is locally of finite variation}, \\
        \infty, & \text{otherwise}.
      \end{cases} \\
         & =
      \begin{cases}
        \inf\{I(x,u,y'): f=F(y')\}, & y = u+f\cdot x \text{ and } x \text{ is locally of finite variation}, \\
        \infty, & \text{otherwise}.
      \end{cases}
    \end{split}
  \end{equation}
  But the contraction principle yields $I(x,u,y) = \inf_f \{J(x,u,y,f)\}$. Hence, if $x$ is of infinite variation, then by \eqref{rate-J}, $J(x,u,y,f) = \infty$ and $I(x,u,y) = \infty$.

  On the other hand, using the contraction principle once again, the rate function $J$ is
  \begin{equation}\label{rate-J-2}
    J(x,u,y,f) =
    \begin{cases}
      I(x,u,y), & f = F(y), \\
      \infty, & \text{otherwise}.
    \end{cases}
  \end{equation}
  Suppose $x$ is locally of finite variation but $y\ne u+F(y)\cdot x$, we will prove that $J(x,u,y,f) = \infty$ and so $I(x,u,y) = \infty$. If $y \ne u+f\cdot x$, then \eqref{rate-J} yields $J(x,u,y,f) = \infty$. If $y = u+f\cdot x$, then $f \ne F(y)$, and $J(x,u,y,f) = \infty$ follow from \eqref{rate-J-2}.

  Suppose now $I(x,u,y)<\infty$. Then the previous arguments yield that $x$ is locally of finite variation and $y = u+F(y)\cdot x$. Again by the contraction principle, $I'(x,u) = \inf_{y'}\{I(x,u,y')\}$, and obviously $I'(x,u) \le I(x,u,y)$. If $I'(x,u) <I(x,u,y)$, then there exists $y'$ such that $I(x,u,y')<I(x,u,y) <\infty$. Hence, $y' = u+F(y')\cdot x$, which yields $y = y'$ by the uniqueness. Therefore, we have $I(x,u,y)=I'(x,u)$ in this case. The representation \eqref{rate-func} follows.
\end{proof}

\section{L\'evy-It\^o SDEs}\label{sec-5}

For each $\e>0$, let $X^\e$ be a quasi-left-continuous $d$-dimensional semimartingale on $(\Omega^\e,\F^\e,\P^\e)$ with characteristics $(B^\e,C^\e,\nu^\e)$ associated to a truncation function $h_b(x) = x\ind_{\{|x|\le b\}}$, $b>0$, and with the measure $\mu^\e$ its jump measure. Then the following representations hold:
$$X^\e = X^\e_0 + B^\e + X^{\e,c} + h_b * (\mu^\e-\nu^\e) + (x-h_b(x)) * \mu^\e,$$
where $X^{\e,c}$ is the continuous martingale part of $X^\e$ so that $C^\e = \langle X^{\e,c}, X^{\e,c}\rangle$. 
We consider the following SDEs of L\'evy-It\^o-type,
\begin{equation}\label{sde-2}
  \begin{split}
    Y^\e_t =&\ U^\e_t + \int_0^t F_1(Y^\e_{s-}) dB^\e_s + \int_0^t F_2(Y^\e_{s-}) dX^{\e,c}_s \\
       &\ + \int_0^t\int_{0<|x|\le b} F_3(Y^\e_{s-})x (\mu^\e-\nu^\e)(dx,ds) + \int_0^t\int_{|x|> b}  F_3(Y^\e_{s-})x \mu^\e(dx,ds).
  \end{split}
\end{equation}
If we let $F(y)=(F_1(y),F_2(y),F_3(y))$ and
\begin{equation*}
  \hat X^\e = (B^\e, X^{\e,c}, h_b * (\mu^\e-\nu^\e) + (x-h_b(x)) * \mu^\e)^\top.
\end{equation*}
Then we can rewrite \eqref{sde-2} into the form of \eqref{sde} as following
\begin{equation}\label{sde-3}
  Y^\e = U^\e + F(Y^\e_-) \cdot \hat X^\e.
\end{equation}
Since each $X^\e$ is a quasi-left-continuous, $\Delta B^\e = 0$ for all $\e$. Then it is easy to verify that each $\hat X^\e$ is a $3d$-dimensional semimartingale with jump measure $\hat\mu^\e(dx,dt) = \mu^\e(dx_3,dt)$, 
where $x=(x_1,x_2,x_3)\in\R^{3d}$. The same relation holds for the L\'evy system $\hat\nu^\e$ of $\hat X^\e$ and $\nu^\e$.
Denote by $\hat B^\e$ and $\hat C^\e$ the first and second characteristics of $\hat X^\e$ associated to $h_b$. Then
\begin{equation*}
  \hat B^\e = (B^\e,0,0)^\top,\quad \hat C^\e = (0,C^\e,0)^\top.
\end{equation*}
Applying Theorem \ref{LDP-special} to the equation \eqref{sde-3}, we have
\begin{corollary}\label{LDP-Ito}
  Let $F_1, F_2, F_3$ be bounded Lipschitz functions, with both Lipschitz constants and themselves bounded by $c>0$. Assume there exists $0<r<\frac{1}{4c}$ such that for each $t>0$, the following three real-valued families are all exponentially tight,
  \begin{equation}\label{three-families-2}
    \{V(B^\e)_t\}_{\e>0}, \quad \left\{\frac{1}{\e} C^\e_t\right\}_{\e>0}, \quad \left\{ \e \int_{\Ro d} \exp\left(\frac{|x|}{\e r}\vee 1 \right) \nu^\e(dx,[0,t]) \right\}_{\e>0}, 
  \end{equation}
  For each $\e>0$, let $Y^\e$ be the solution of \eqref{sde-2}. If the family $\{(\hat X^\e,U^\e)\}_{\e>0}$ satisfies the LDP with good rate function $I'$ on $\D_{3d+n}$, then the family $\{Y^\e\}_{\e>0}$ satisfies the LDP, with good rate function \eqref{rate-func-2} and with $F=(F_1,F_2,F_3)$.
\end{corollary}

\begin{example}\label{examp}
  Consider for each $0<\e\le1$ an additive process $X^\e = \{X^\e_t\}_{t\ge0}$  (\cite[Section 9]{Sat99} or \cite[Section II.4]{JS13}) given by
  \begin{equation*}
    X^\e_t = b^\e_t + W^\e_t + L^\e_t,
  \end{equation*}
  where $b^\e: [0,\infty)\to\R^d$ is a deterministic continuous functions, $W^\e=\{W^\e_t\}_{t\ge0}$ is a $d$-dimensional (nonstandard) Wiener process with quadratic covariation $\langle W^\e,W^\e\rangle = c^\e$ and $c^\e$ is a deterministic continuous symmetric nonnegative-definite $d\times d$ matrix-valued functions on $[0,\infty)$, $L^\e$ is a  $d$-dimensional pure jump process given by
  \begin{equation}\label{levy}
    L^\e_t := \e\int_0^t\int_{0<|x|\le 1} x \tilde N^{\e^{-1}}(dx,ds) + \e\int_0^t\int_{|x|>1} x N^{\e^{-1}}(dx,ds),
  \end{equation}
  where $N^{\e^{-1}}$ is a (inhomogeneous) Poisson random measure, independent of $W$, with deterministic intensity measure $\e^{-1} \nu$ and compensated random measure $\tilde N^{\e^{-1}}$, and $\nu$ is a given $\sigma$-finite measure on $(\Ro d) \times[0,\infty)$ satisfying $\nu((\Ro d) \times\{t\})=0$ and $\int_{\Ro d} (|x|^2\wedge 1) \nu(dx,[0,t])<\infty$ for all $t\ge0$. 
  When $\e=1$, we denote $L := L^1$ for simplicity. Then $\nu$ is exactly the L\'evy measure of $L$. It is easy to see that $L^\e \stackrel{\mathtt d}{=} \e L_{\cdot/\e}$. 
  To identify the characteristics (or generating triplet) of each $X^\e$, we define a new family of Poisson random measures $\mathcal N^\e$ by
  \begin{equation*}
    \mathcal N^\e(A\times[0,t]) = N^{\e^{-1}}((A/\e)\times[0,t]), \quad\text{ for } A\in\B(\Ro d), t>0,
  \end{equation*}
  where $A/\e:=\{x/\e:x\in A\}$. Then $\mathcal N^\e$ has intensity measure $\e^{-1} \nu(\e^{-1}dx,dt)$. Denote the associated compensated random measure of $\mathcal N^\e$ by $\tilde{\mathcal N}^\e$. We use the change of variable $y=\e x$ to rewrite \eqref{levy} as following,
  \begin{equation*}
    \begin{split}
      L^\e_t &= \int_0^t\int_{0<|y|\le \e} y \tilde  N^{\e^{-1}}(dy/\e,ds) + \int_0^t\int_{|y|>\e} y  N^{\e^{-1}}(dy/\e,ds) \\
      &= \int_0^t\int_{0<|y|\le \e} y \tilde{\mathcal N}^\e(dy,ds) + \int_0^t\int_{|y|>\e} y \mathcal N^\e(dy,ds) \\
      &= \int_0^t\int_{0<|y|\le 1} y \tilde{\mathcal N}^\e(dy,ds) + \int_0^t\int_{|y|>1} y \mathcal N^\e(dy,ds) + \int_{1<|x|\le 1/\e} x \nu(dx,[0,t]).
    \end{split}
  \end{equation*}
  Hence the characteristics $(B^\e,C^\e,\nu^\e)$ of $X^\e$ associated to the truncation function $h_1$ is deterministic and given by
  \begin{equation*}\left\{
    \begin{array}{l}
      B^\e_t = b^\e_t + \int_{1<|x|\le 1/\e} x \nu(dx,[0,t]), \\
      C^\e_t = c^\e_t, \\
      \nu^\e(dx,dt) = \e^{-1} \nu(\e^{-1}dx, dt).
    \end{array} \right.
  \end{equation*}
  To ensure the conditions in Corollary \ref{LDP-Ito}, we make the following assumptions:
  \begin{itemize}
    \item There exists a function $K:[0,\infty)\to[0,\infty)$ such that for all $0<\e\le1$ and $t\ge0$,
        $$|V(b^\e)_t| \vee \textstyle{\frac{1}{\e}} |c^\e_t| \le K(t).$$
    \item There exists an $r>0$, such that for all $t\ge0$ and $\lambda>0$,
        \begin{equation}\label{con-exp-int}
          \nu(\{0<|x|\le r\},[0,t])<\infty, \quad\text{and}\quad \int_{|x|>r} e^{\lambda |x|} \nu(dx,[0,t])<\infty.
        \end{equation}
    \item The family $\{W^\e\}$ satisfies an LDP in the space $\C_d:=\C(\R_+,\R^d)$ with uniform topology.
  \end{itemize}
  A typical example for $W^\e$ to fulfill the third assumption is that $W^\e = \sqrt\e W$, where $W$ is a $d$-dimensional standard Brownian motion (e.g., \cite[Theorem 5.2.3]{DZ98}). Since
  \begin{equation*}
    |V(B^\e)_t| \vee \frac{1}{\e} |C^\e_t| \le K(t) + \int_{|x|>1} |x| \nu(dx,[0,t]) \le K(t) + \int_{|x|>1} e^{|x|} \nu(dx,[0,t]),
  \end{equation*}
  and
  \begin{equation*}
    \begin{split}
      \e \int_{\Ro d} \exp\left(\frac{|x|}{\e r}\vee 1 \right) \nu^\e(dx,[0,t]) &= \int_{\Ro d} \exp\left(\frac{|y|}{ r}\vee 1 \right) \nu(dy,[0,t]) \\
      &= e \nu(\{0<|y|\le r\},[0,t]) + \int_{|y|>r} e^{|y|/r} \nu(dy,[0,t]),
    \end{split}
  \end{equation*}
  where the right hand sides are both deterministic and independent of $\e$, the three families in \eqref{three-families-2} are exponentially tight for each $t>0$. On the other hand, 
  since almost all sample paths of $W^\e$ lies in $\C_d$ and the uniform topology on $\D_d$ is finer than the Skorokhod topology, the family $\{W^\e\}$ also satisfies an LDP in $\D_d$ with Skorokhod topology. The assumption on $\nu$ yields an LDP for the family $\{L^\e\}$, referring to \cite[Theorem 1.2]{DeA93}. Thus, the independence of $W$ and $L^\e$ yields that the family $\{\hat X^\e = (b^\e, W^\e, L^\e)\}$ satisfies an LDP (cf. \cite[Exercise 4.2.7]{DZ98}). Consider the following family of L\'evy-driven SDEs
  \begin{equation*}
    \begin{split}
      Y^\e_t =&\ Y_0^\e + \int_0^t F_1(Y^\e_{s}) d b^\e_s + \int_0^t F_2(Y^\e_{s}) dW^\e_s \\
         &\ +\e \int_0^t\int_{0<|x|\le 1} F_3(Y^\e_{s-})x\tilde{N}^{\e^{-1}}(dx,ds) +\e \int_0^t\int_{|x|> 1}  F_3(Y^\e_{s-})x N^{\e^{-1}}(dx,ds),
    \end{split}
  \end{equation*}
  where $F_1, F_2, F_3$ are bounded Lipschitz functions, the initial value $Y_0^\e$ is an $\R^d$-valued random variable and independent of $W$ and $L^\e$ for each $\e$. We suppose the family of initial values $\{Y_0^\e\}$ satisfies an LDP. Then the family $\{(Y_0^\e,\hat X^\e)\}$ also satisfies an LDP, again by the independence. Therefore, Corollary \ref{LDP-Ito} yields that the solution family $\{Y^\e\}$ satisfies an LDP, with a good rate function.
\end{example}

\begin{remark}
  (i). We remark that the assumption $\int_{|x|>r} e^{\lambda |x|} \nu(dx,[0,t])<\infty$ for all $t\ge0$ and $\lambda>0$ is equivalent to (see \cite[Theorem 25.3]{Sat99})
  \begin{equation*}
    \E^\e( e^{\lambda L_1} ) <\infty, \quad \text{for all } \lambda>0,
  \end{equation*}
  which is the \emph{exponential integrability} condition found in \cite{DeA93,Kuh14}.

  (ii). Generalizations of Example \ref{examp} and Corollary \ref{LDP-Ito} to L\'evy-It\^o SDEs or processes with $\e$-dependent coefficients can be found in \cite{Puh04,Log12}. The generalized counterparts of condition \eqref{con-exp-int} are \cite[Section 4, ($\tilde{\text{P}}$) and ($\widetilde{\text{SE}}$)]{Puh04} and \cite[Eq.~(2)]{Log12} respectively.
\end{remark}

\paragraph{Acknowledgements.}
The research of J. Duan was partly supported by the NSF grant 1620449. The research of Q. Huang was partly supported by FCT, Portugal, project PTDC/MAT-STA/28812/2017. We would like to thank the reviewers for their thoughtful comments and efforts towards improving our manuscript.

\begin{appendices}

\section{More on UET property}\label{app}

In this section, we will present more characterizations, both sufficient and necessary, for the UET property. Firstly, the converse of the first statement in Lemma \ref{exp-tight-bv-conti} holds when the processes are predictable, as we will show in the following lemma.

\begin{lemma}
  Let $\{B^\e\}_{\e>0}$ be a family of one-dimensional predictable processes starting at $0$ with locally finite variation. If the family $\{B^\e\}_{\e>0}$ is UET, then for each $t>0$, the family $\{V(B^\e)_t\}_{\e>0}$ is exponentially tight.
\end{lemma}
\begin{proof}
  Note first that for each $\e>0$, there is a predictable set $A_\e$ on $(\Omega^\e,\F^\e,\{\F^\e_t\},\P^\e)$ such that $V(B^\e) = ( \ind_{A_\e} - \ind_{(A_\e)^c}) \cdot B^\e$ and $B^\e = ( \ind_{A_\e} - \ind_{(A_\e)^c}) \cdot V(B^\e)$. Fix $t>0$ and $\delta>0$. For each $\e>0$, by \cite[Lemma I.3.10]{JS13}, there is a stopping time $T^\e$ such that
  \begin{equation}\label{est-2}
    \P^\e(T^\e<t) < \delta^{1/\e}
  \end{equation}
  and $\E^\e(V(B^\e)_{T^\e})<\infty$. The last inequality ensures that we can choose a predictable set $A'_\e$ from the algebra generated by $\{A\times\{0\}: A\in \F^\e_0\} \cup \{A\times (s,r]: A\in \F^\e_s, 0\le s<r\}$, which is a subalgebra generating the predictable $\sigma$-algebra, such that
  \begin{equation}\label{est-3}
    \E^\e((|\ind_{A_\e}-\ind_{A'_\e}| \cdot V(B^\e))_{T^\e})\le a\delta^{1/\e}/2.
  \end{equation}
  The way of choosing the set $A'_\e$ yields $\ind_{A'_\e}-\ind_{(A'_\e)^c} \in \Pred_1^\e$. It is easy to deduce that $V(B^\e) = ( \ind_{A'_\e} - \ind_{(A'_\e)^c}) \cdot B^\e + 2 |\ind_{A_\e}-\ind_{A'_\e}| \cdot V(B^\e)$. Then \eqref{est-2}, \eqref{est-3} and Definition \ref{UET-def} yield for every $0<\e\le\e_0 \wedge 1$,
  \begin{equation*}
    \begin{split}
       &\ [\P^\e(V(B^\e)_t > 2a)]^\e \\
       =&\ \left[\P^\e(V(B^\e)_t > 2a, T^\e<t) + \P^\e(V(B^\e)_t > 2a, T^\e\le t)\right]^\e \\
         \le&\ \delta + \left[\P^\e( (( \ind_{A'_\e} - \ind_{(A'_\e)^c}) \cdot B^\e)_t > a )\right]^\e + \left[\P^\e( 2( |\ind_{A_\e}-\ind_{A'_\e}| \cdot V(B^\e))_{T^\e} > a )\right]^\e \\
         \le&\ 3\delta.
    \end{split}
  \end{equation*}
  The result follows.
\end{proof}


\begin{lemma}
  Let (i), (ii) and (iii) are those statements in Lemma \ref{exp-tight-large-jumps} and \ref{exp-tight-large-jumps-2}. Then we have the equivalence (i)$\Leftrightarrow$(i') and the implications 
  (ii')$\Rightarrow$(ii)$\Rightarrow$(ii'') and (iii')$\Rightarrow$(iii), 
  where (i'), (ii'), (ii'') and (iii') are given by the following: \\
  (i'). For all $\eta>0$, $\lim_{r\to\infty} \limsup_{\e\to0} \e\log \P^\e(\nu^\e(\{|x|>r\}\times[0,t])>\eta^{1/\e}) = -\infty$. \\
  (ii'). For all $r>0$, $\lim_{\eta\to 0} \limsup_{\e\to0} \e\log \P^\e(\nu^\e(\{|x|>r\}\times[0,t])>\eta^{1/\e}) = -\infty$. \\
  (ii''). For all $r>0$, $\lim_{\eta\to \infty} \limsup_{\e\to0} \e\log \P^\e(\nu^\e(\{|x|>r\}\times[0,t])>\eta^{1/\e}) = -\infty$. \\
  (iii'). $\lim_{\eta\to\infty} \lim_{b\to\infty} \limsup_{\e\to0} \e\log \P^\e ( \e\exp(|x|/(\e r))\ind_{\{\e r<|x|\le b\}} * \nu^\e_t >\eta ) = -\infty$.
\end{lemma}

\begin{proof}
  Next, we prove the equivalence (i)$\Leftrightarrow$(i'). To this purpose, we define for each $r>0$ and $\e>0$, $A^{r,\e}_t:=\sum_{0\le s\le t} \ind_{\{|\Delta X^\e_s|>r\}}$. Firstly we note that for all $\{\F^\e_t\}$-stopping times $T$,
  $$\E^\e\left(A^{r,\e}_T\right) = \E^\e\left(\nu^\e(\{|x|>r\}\times[0,T])\right).$$
  Assume (i) holds. Using Lenglart's inequality (see \cite[Lemma I.3.30]{JS13}), we have
  \begin{equation}\label{est-13}
    \begin{split}
      \P^\e\left( \nu^\e(\{|x|>r\}\times[0,t])>\eta^{1/\e} \right) & \le b^{1/\e} + \eta^{-1/\e} \E^\e \left( \sup_{0\le s\le t} \Delta A^{r,\e}_s \right) + \P^\e \left( A^{r,\e}_t \ge (b\eta)^{1/\e} \right) \\
         & \le b^{1/\e} + (\eta^{-1/\e}+1) \P^\e\left( \sup_{0\le s\le t}|\Delta X^\e_s| > r \right),
    \end{split}
  \end{equation}
  since $0\le\Delta A^{r,\e}\le1$ and $\{ \sup_{0\le s\le t} \Delta A^{r,\e}_s >0 \} = \{A^{r,\e}_t>0\} = \{ \sup_{0\le s\le t}|\Delta X^\e_s| > r \}$. Thus,
  \begin{equation*}
    \begin{split}
      \limsup_{\e\to0} \e\log \P^\e\left( \nu^\e(\{|x|>r\}\times[0,t])>\eta^{1/\e} \right) \le &\ (\log b) \vee \bigg[ ((-\log\eta)\vee 0) \\
         & \qquad\qquad\  + \limsup_{\e\to0} \e\log \P^\e\left( \sup_{0\le s\le t}|\Delta X^\e_s| > r \right) \bigg],
    \end{split}
  \end{equation*}
  which yields (i') by letting first $r\to\infty$ and then $b\to0$. For (i')$\Rightarrow$(i), we use Lenglart's inequality again to deduce
  \begin{equation}\label{est-15}
    \P^\e\left( \sup_{0\le s\le t}|\Delta X^\e_s| > a \right) = \P^\e \left( A^{a,\e}_t \ge 1 \right) \le \eta^{1/\e} + \P^\e\left( \nu^\e(\{|x|>a\}\times[0,t])>\eta^{1/\e} \right).
  \end{equation}
  Then the result follows by taking $\limsup_{\e\to0} \e\log$ on both sides and then letting $a\to\infty$ and $\eta\to0$ successively.

  The approach of proving (ii')$\Rightarrow$(ii)$\Rightarrow$(ii'') is similar. The key is using Lenglart's inequality to get
  \begin{equation}\label{est-14}
    \P^\e \left( A^{r,\e}_t \ge a \right) \le \frac{\eta^{1/\e}}{a} + \P^\e\left( \nu^\e(\{|x|>r\}\times[0,t])>\eta^{1/\e} \right),
  \end{equation}
  \begin{equation*}
    \begin{split}
      \P^\e\left( \nu^\e(\{|x|>r\}\times[0,t])>\eta^{1/\e} \right) & \le b \eta^{-1/\e} + \eta^{-1/\e} \E^\e \left( \sup_{0\le s\le t} \Delta A^{r,\e}_s \right) + \P^\e \left( A^{r,\e}_t \ge b \right) \\
      & \le (b+1) \eta^{-1/\e} + \P^\e \left( A^{r,\e}_t \ge b \right).
    \end{split}
  \end{equation*}
  Take $\limsup_{\e\to0} \e\log$ on both sides of above inequalities. Then let $\eta\to0$, $a\to\infty$ for the first inequality and let $\eta\to\infty$, $b\to\infty$ for the second. The results follow easily.

  We now prove the implication (iii')$\Rightarrow$(iii). Since $\sum_{0\le s\le t} |\Delta X^\e_s|\ind_{\{\e r<|\Delta X^\e_s|\le b\}} = |x|\ind_{\{\e r<|x|\le b\}} * \mu^\e_t$, we use Lemma \ref{estimates} to get
  \begin{equation*}
    \P^\e\left( \sum_{0\le s\le t} |\Delta X^\e_s|\ind_{\{\e r<|\Delta X^\e_s|\le b\}} > a \right) \le e^{(\eta-a/r)/\e} + \P^\e\left( \e\exp(|x|/(\e r))\ind_{\{\e r<|x|\le b\}} * \nu^\e_t >\eta \right).
  \end{equation*}
  Take $\limsup_{\e\to0} \e\log$ on both sides and then let $b\to\infty$, $a\to\infty$ and $\eta\to\infty$ successively. The result follows.
\end{proof}

\begin{remark}
  From the proof, one can see that (ii') implies much stronger conditions than (i) and (ii). Indeed, by \eqref{est-15}, (ii') implies
  \begin{equation}\label{C-exp-tight}
    \limsup_{\e\to0} \e\log \P^\e\left( \sup_{0\le s\le t}|\Delta X^\e_s|\ge a \right) = -\infty, \quad\text{for all } a>0.
  \end{equation}
  If \eqref{C-exp-tight} holds for all $t>0$, then one say that the family $\{X^\e\}_{\e>0}$ is $\C$-\emph{exponentially tight} (see \cite[Definition 4.12]{FK06}). Here the symbol $\C$ indicates the continuity. More precisely, we denote by $\C_d:=\C(\R_+;\R^d)$. On the other hand, by \eqref{est-14}, (ii') implies
  \begin{equation*}
    \limsup_{\e\to0} \e\log \P^\e\left( \sum_{0\le s\le t} \ind_{\{|\Delta X^\e_s|>r\}} > a \right) = -\infty, \quad\text{for all } r>0 \text{ and } a>0.
  \end{equation*}
  Moreover, from the inequalities \eqref{est-13} and \eqref{est-14}, it is easy to see that (i) implies
  \begin{equation*}
    \lim_{r\to\infty} \limsup_{\e\to0} \e\log \P^\e\left( \sum_{0\le s\le t} \ind_{\{|\Delta X^\e_s|>r\}} > a \right) = -\infty, \quad\text{for all } a>0.
  \end{equation*}
\end{remark}

The following lemma presents some necessary conditions for the UET property of $\{X^\e\}_{\e>0}$. 
\begin{lemma}
  Let $\{X^\e\}_{\e>0}$ be a family of one-dimensional semimartingales starting at $0$ satisfying UET condition. Then for each $t>0$, the two families $\{\sup_{0\le s\le t}|X^\e_s|\}_{\e>0}$ and $\{[X^\e,X^\e]_t\}_{\e>0}$ are exponentially tight.
\end{lemma}
\begin{proof}
  Fix $t>0$. The exponential tightness of $\{\sup_{0\le s\le t}|X^\e_s|\}$ follows from Definition \ref{UET-def} by taking $H\equiv 1$. For each $\e>0$, $n\in\N_+$, define a process
  \begin{equation*}
    Q^{\e,n}_t := \sum_{i=1}^n (X^\e_{it/n} - X^\e_{(i-1)t/n})^2.
  \end{equation*}
  By the construction of quadratic variation (see \cite[Theorem I.4.47]{JS13}), for each $\e>0$, the process $Q^{\e,n}$ converges in probability $\P^\e$ to $[X^\e,X^\e]$, uniformly on compact intervals, as $n\to\infty$. Then for all $a>0$,
  \begin{equation*}
    \P^\e( [X^\e,X^\e]_t>a ) \le \liminf_{n\to\infty}\P^\e( Q_t^{\e,n}>a ).
  \end{equation*}
  Set $H^{\e,n} = -2\sum_{i=1}^n X^\e_{(i-1)t/n} \ind_{((i-1)t/n, it/n]}$, then
  \begin{equation}\label{Q}
    Q^{\e,n} = |X^\e|^2 + H^{\e,n}\cdot X^\e.
  \end{equation}
  Since for each $b>0$,
  \begin{equation*}
    \left\{ \sup_{0\le s\le t}|X^\e_s| \le b \right\} \subset \{ |H^{\e,n}| \le 2b \},
  \end{equation*}
  using \eqref{Q} we have for all $a> b^2$,
  \begin{equation*}
    \begin{split}
      \P^\e( Q^{\e,p}_t>a ) & \le \P^\e\left( \sup_{0\le s\le t}|X^\e_s| > b \right) + \P^\e\left( |(H^{\e,n}\cdot X^\e)_t| > a-b^2,\sup_{0\le s\le t}|X^\e_s| \le b\right) \\
         & \le \P^\e\left( \sup_{0\le s\le t}|X^\e_s| > b \right) + \sup_{H\in\Pred_1^\e} \P^\e\left( \left| \left( H \cdot X^\e \right)_t \right| > \frac{a-b^2}{2b} \right).
    \end{split}
  \end{equation*}
  Taking $\e\log$ on both sides, and letting first $a\to\infty$ and then $b\to\infty$, the exponential tightness of $\{[X^\e,X^\e]_t\}$ follows from that of $\{\sup_{0\le s\le t}|X^\e_s|\}$ and the UET property of $\{X^\e\}$.
\end{proof}

\section{$\C$-exponential tightness}\label{app-2}

We say that a rate function $I$ on $\D_d$ is \emph{quasi-left continuous} at $t\ge0$ if $I(x) = \infty$ for all $x\in\D_d$ with $\Delta x(t) \ne 0$. Recall that $\C_d:=\C(\R_+;\R^d)$ denotes the space of all $\R^d$-valued continuous functions on $\R_+$. If $I=\infty$ on $\D_d\setminus\C_d$, then it is quasi-left continuous at all $t\in[0,\infty)$. The following criterion for quasi-left continuity of rate functions is a counterpart of \cite[Theorem 4.13]{FK06}. We will omit the proof.
\begin{lemma}
  Let a family of probability measure $\{\mu_\e\}_{\e>0}$ on $\D_d$ be exponentially tight. Let $A\subset [0,\infty)$. Then each rate function $I$ that gives a large deviation principle for a subsequence $\{\mu_{\e_k}\}_{k=1}^\infty$ with $\e_k\to 0$ as $k\to\infty$ is quasi-left continuous at all $t\in A$ if and only if for every $T>0$ and $\eta>0$,
  \begin{equation*}
    \limsup_{\e\to0} \e \log \mu_\e\left( x\in \D_d: \sup_{t\in A\cap[0, T]}|\Delta x(t)| \ge\eta \right) = -\infty.
  \end{equation*}
\end{lemma}

Using this lemma and the definition of $\C$-exponential tightness recalled in \eqref{C-exp-tight}, we can see that a family of c\`adl\`ag processes $\{X^\e\}_{\e>0}$ is $\C$-exponentially tight if and only if each rate function $I$ that gives a large deviation principle for a subsequence $\{X^{\e_k}\}_{k=1}^\infty$ with $\e_k\to 0$ as $k\to\infty$ is quasi-left continuous at all $t\in [0,\infty)$.

\begin{corollary}
  If the rate function $I'$ is quasi-left continuous at $t_0\in [0,\infty)$, then $I$ and $I^\flat$ defined in \eqref{rate-func} and \eqref{rate-func-2} are also quasi-left continuous at $t_0$. In particular, if the family $\{(X^\e,U^\e)\}_{\e>0}$ is $\C$-exponentially tight, so are $\{(X^\e,U^\e,Y^\e)\}_{\e>0}$ and $\{Y^\e\}_{\e>0}$.
\end{corollary}
\begin{proof}
  Obviously, the quasi-left continuity of $I^\flat$ is implied by that of $I$. The second statement is trivial by the preceding argument. So we only need to show the quasi-left continuity of $I$. Let $\Delta (x,u,y)(t_0) \ne 0$. It is only needed to consider the case that $y = u+F(y)\cdot x$, and $x$ is of finite variation. If $\Delta (x,u)(t_0) \ne 0$, then $I(x,u,y) = I'(x,u) = \infty$. If not, then $\Delta y(t_0) \ne 0$, which still yields $\Delta (x,u)(t_0) \ne 0$ since $\Delta y = \Delta u + F(y)\Delta x$. We are done.
\end{proof}

\end{appendices}

\footnotesize{
\bibliographystyle{MyStyle-plainnat}
\bibliography{LDP}
}

\end{document}